\documentclass[a4paper, 12pt, pdftex,reqno
]{amsart}
\usepackage[hang, flushmargin, bottom]{footmisc} 

\usepackage[hyperindex,breaklinks,pdftex, bookmarks]{hyperref}
\usepackage[color, 
 final  
]{showkeys}

\usepackage[all,  curve, frame, arc, color]{xy}

\usepackage{mathtools}

\usepackage{helvet}         
\usepackage{courier}        
\usepackage{graphicx}        
\usepackage{multicol}        

\usepackage{diagbox}

\usepackage{array,colortbl
}

\makeindex             

\usepackage[utf8x]{inputenc}
\usepackage[english]{babel}
\usepackage{amsmath} 
\usepackage{amscd}
\usepackage{amsfonts}
\usepackage{amssymb}
\usepackage{amsthm} 
\usepackage{mathabx}
\usepackage{mathrsfs}
\usepackage{graphicx}
\usepackage{textcomp}
\usepackage{url}
\usepackage{amsopn,latexsym,amscd}
\usepackage
{todonotes}
\usepackage{fullpage}
\usepackage{graphicx}

\usepackage{verbatim}

\usepackage{enumitem}

\makeatletter
\newcommand{\eqnum}{\refstepcounter{equation}\textup{\tagform@{\theequation}}}
\makeatother

\newtheorem{thm}{Theorem}[section]
\newtheorem{lem}[thm]{Lemma}  

\newtheorem{prop}[thm]{Proposition}
\newtheorem{cor}[thm]{Corollary}

\renewcommand*{\theRthm}{\Alph{Rthm}}

\theoremstyle{definition}
\newtheorem{df}[thm]{Definition}

\newtheorem{rem}[thm]{Remark}
\newcommand{\Z}{\mathbb{Z}}
\newcommand{\C}{{\mathbb{C}}}
\newcommand{\R}{{\mathbb{R}}}
\newcommand{\Q}{{\mathbb{Q}}}
\newcommand{\N}{{\mathbb{N}}}

\newcommand{\F}{\mathbb{F}}
\newcommand{\md}{\mathrm{M}}
\newcommand{\pp}{{\mathbb{P}}}

\newcommand{\cB}{\mathcal{B}}
\newcommand{\G}{{\mathcal{G}}}
\newcommand{\rk}{{\operatorname{rk}}}
\newcommand{\codim}{{\operatorname{codim}}}
\newcommand{\Aut}{{\operatorname{Aut}}}
\newcommand{\out}{{\operatorname{Out}}}
\newcommand{\into}[0]{\hookrightarrow}
\newcommand{\id}{\mathrm{id}}

\newcommand{\Filt}{\mathcal{F}}

\newcommand{\Hom}{\operatorname{Hom}}

\newcommand{\qbin}[2]{ \left[ \! \! \! \begin{array}{c} #1 \\ #2 \end{array} \! \! \! \right] }

\newcommand{\Ga}[1]{\operatorname{G_{A_{#1}}}}
\newcommand{\Gb}[1]{\operatorname{G_{B_{#1}}}}
\newcommand{\Gbn}{\operatorname{G_{B_n}}}
\newcommand{\Br}{\operatorname{Br}}
\newcommand{\perm}{\mathfrak{S}}

\newcommand{\disk}{\mathrm{D}}
\newcommand{\ddiskP}{\widetilde{\disk \setminus \P}}
\newcommand{\conf}{\operatorname{C}}
\newcommand{\confn}{\conf_n}
\newcommand{\confm}[1]{\operatorname{C_{1,#1}}}
\newcommand{\confmn}{\confm{n}}
\newcommand{\confmd}[1]{\widetilde{\operatorname{C_{1,#1}}}}
\newcommand{\confmdn}{\confmd{n}}

\newcommand{\sym}[1]{\operatorname{Sp({#1})}}

\newcommand{\totsp}{\mathrm{E}}

\renewcommand{\P}{\mathrm{P}}

\newcommand{\compl}{\disk \setminus }
\newcommand{\ddisk}{def}
\newcommand{\surf}{\Sigma}
\renewcommand{\ring}{R}

\newcommand{\st}{\Gamma}
\newcommand{\pmu}{{\pm 1}}

\newcommand{\DA}{\partial}
\newcommand{\ppartial}{\overline{\partial}}
\newcommand{\DB}{\ppartial}

\newcommand{\scsc}[2]{ \substack{
		{}_{#1}\\
		{}_{#2}
	} }

\newcommand{\gp}{\widetilde{\gamma}}
\newcommand{\gpu}{\widetilde{\gamma}_1}

\newcommand{\bool}[1]{\mathcal{P}[#1]}

\DeclareMathOperator{\gs}{\gamma}
\DeclareMathOperator{\gsu}{\gamma_1}
\DeclareMathOperator{\coker}{\mathrm{coker}}

\newcommand{\MMM}{\mathrm{M}}
\newcommand{\MMMp}{\widetilde{\MMM}}
\newcommand{\MM}[1]{\operatorname{\MMM}(#1)}
\newcommand{\MMp}[1]{\MMMp(#1)}
\newcommand{\Id}{\mathrm{Id}}

\newcommand{\stab}{\mathrm{st}}
\newcommand{\gstab}{\mathrm{gst}}

\newlength\Origarrayrulewidth

\newcommand{\chl}{\cellcolor{lightgray!30}}

\begin{document}

\title{
Homology of the family of hyperelliptic curves}

\author{Filippo Callegaro and Mario Salvetti}
\date{\today 
}
\address[F. Callegaro]{
Dipartimento di Matematica, University of Pisa, Italy.
}
\email{callegaro@dm.unipi.it}

\address[M. Salvetti]{Dipartimento di Matematica, University of Pisa, Italy.} 
\email{salvetti@dm.unipi.it}

\begin{abstract}
Homology of braid groups and Artin groups can be related to the study of spaces of curves.
We completely calculate the integral homology of the family of smooth curves of genus $g$ with one boundary component, that are double coverings of the disk ramified over $n = 2g + 1$ points.
The main part of such homology is described by the homology of the braid group with coefficients in a symplectic representation, namely the braid group $\Br_n$ acts on the first homology group of a genus $g$ surface via Dehn twists. Our computations shows that such groups have only $2$-torsion. We also investigate stabilization properties and provide Poincar\'e series, both for unstable and stable homology.
\end{abstract}

\maketitle

\section{Introduction}
 In this paper we consider 
 the family of hyperelliptic curves 
$$
\totsp_{n} := \{(\P, z,y ) \in \conf_n  \times \disk \times \C| y^2 = (z-x_1)\cdots(z-x_n) \}.
$$ 
where $\disk$ is the unit open disk in $\C, $ \  $\conf_n$ is the configuration space of $n$ distinct unordered points in $\disk$ and $\P=\{x_1,\dots,x_n\}\in \conf_n.$  
Each curve $\surf_n$ in the family is a ramified double covering of the disk $\disk$ and there is a fibration  $\pi:\totsp_n\to \conf_n$ which takes  $\surf_n$ onto its set of ramification points. Clearly $\totsp_n$ is a universal family over the Hurwitz space $H^{n,2}$ (for precise definitions see \cite{fulton}, \cite{evw}).

The aim of this paper is to compute the integral homology of the space $\totsp_n.$ 
The rational homology of $\totsp_n$ is known, having been computed in \cite{chen} by using \cite{cms_tams}.
The bundle $\pi:\totsp_n\to \conf_n$ has a global section, so $H_*(E_n)$ splits into a direct sum $H_*(\conf_n)\oplus H_*(\totsp_n,\conf_n)$ and by the Serre spectral sequence $H_*(\totsp_n,\conf_n)=H_{*-1}(\Br_n;H_1(\surf_n)).$ 
We use here that $\conf_n$ is a classifying space for the braid group $\Br_n.$ The action of the braid group over the homology of the surface is geometrical: the braid group embeds (see \cite{per_van_92, waj_99}) into the mapping class group of the surface (with one or two boundary components according to $n$ odd or even respectively) by taking the standard generators into particular Dehn twists.
In this paper we actually compute the homology of this symplectic representation of the braid groups. Since the homology of the braid groups is well-know (see for example \cite{fuks, vain, cohen}) we obtain a description of the homology of $\totsp_n.$ 
It would be natural to extend the computation to the homology of the braid group $\Br_n$ with coefficients in the symmetric powers of  $H_1(\surf_n)$. In the case of $n=3$ a complete computation (in cohomology) can be found in \cite{ccs}.

Some experimental computations given in \cite{msv2012} have led us to conjecture that $H_{*}(\Br_n; H_1(\surf_n;\Z))$  is only $2$-torsion for odd $n.$ 

Our main results are the following 

\begin{thm} (see Theorem \ref{th:no4tor}, \ref{thm:poincare})

For odd $n:$
\begin{enumerate}
\item  the integral homology $H_{i}(\Br_n; H_1(\surf_n;\Z))$ has only $2$-torsion. 
\item the rank of $H_i(\Br_n;H_1(\surf_n;\Z))$ as a $\Z_2$-module is the coefficient of $q^it^n$ in the series
 $$
\widetilde{P}_2(q,t)=\frac{qt^3}{(1-t^2q^2)} \prod_{i \geq 0} \frac{1}{1-q^{2^i-1}t^{2^i}}
$$
In particular the series $\widetilde{P}_2(q,t)$ is 
the Poincar\'e series of the homology group  
$$\bigoplus_{n \mbox{\scriptsize odd}}H_*(\Br_{n};H_1(\surf_n;\Z))$$
as a $\Z_2$-module. 
\end{enumerate}
\end{thm}

\begin{thm} (see Theorem \ref{thm:stabilization}, \ref{thm:stablepoincare})
Consider homology with integer coefficients. 
\begin{enumerate}
\item The homomorphism 
$$
H_i(\Br_n; H_1(\surf_n)) \to H_i(\Br_{n+1}; H_1(\surf_{n+1}))
$$
is an epimorphism for $i \leq \frac{n}{2}-1 $
and an isomorphism for $i < \frac{n}{2}-1$.

\item For $n$ even $H_i(\Br_n; H_1(\surf_n))$ has no $p$ torsion (for $p > 2$) when $\frac{pi}{p-1}+3 \leq n$ and no free part for $i+3 \leq n$. In particular for $n$ even,  when $\frac{3i}{2}+3 \leq n$ the group $H_i(\Br_n; H_1(\surf_n))$ has only $2$-torsion.
\item 
The Poincar\'e polynomial of the stable homology $H_i(\Br_n;H_1(\surf_n;\Z))$ as a $\Z_2$-module is the following:
$$
P_2(\Br;H_1(\Sigma))(q) = \frac{q}{1-q^2} \prod_{j \geq 1} \frac{1}{1-q^{2^j-1}}
$$
\end{enumerate}
\end{thm}

We also find unstable free components in the top and top-1 dimension for even $n$ (Theorem \ref{thm:unstable}). 

The main tools that we use are the following. 

First, we use here some of the geometrical ideas in \cite{bianchi}, where  the author shows that the $H_*(\Br_n;H_1(\surf_n))$ is at most $4$-torsion using some exact sequences obtained from a Mayer-Vietoris decomposition. 

Second, we identify the homology groups which appear in the exact sequences with local homology groups of the configuration space $\confm{n}$ of $n+1$ points with one distinguished point. Such spaces are the classifying space of the  Artin groups of type $\mathrm{B},$ so we can use some of  the homology computations given in \cite{calmar}: our results heavily rely on these computations and we collect most of those we need in Section \ref{sec:homol_artin}.

Some explicit computations are provided in Table \ref{tab:conti}.

\setcounter{tocdepth}{1}
\tableofcontents

\section{Notations and preliminary definitions}

Along all this paper, when not specified, the homology is understood to be computed with constant coefficients over a ring $\ring$. 
Let $p$ be a prime or $0$ and let $\F$ be a field of characteristic $p$. 
We write $\md$ for the $\ring$-module of Laurent series $\ring[t, t^{-1}] $. 


Let $\Ga{n-1}  = \Br_n$ be the classical braid group on $n$ strands and let $\Gb{n}$ be the 
Artin group of type $\mathrm{B}$.

We write $\confn$ for the configuration space of $n$ unordered points in the unitary disk  $\disk:= \{z \in \C | \, |z|  < 1\}$.
A generic element of $\confn$ is an unordered set of $n$ distinct points $\P = \{x_1, \ldots, x_n\} \subset \disk$. In particular $\conf_1 = \disk$. The fundamental group of $\conf_n$ is the classical braid group on $n$ strands $\Ga{n-1}  = \Br_n$ and we recall that the space $\conf_n$ is a $K(\Br_n,1)$ (see \cite{fa_neu_62}).

Given an element $\P \in \confn$, we can consider the set of points $$\surf_n:= \{(z,y) \in  \disk \times \C | y^2 = (z-x_1)\cdots(z-x_n) \}.$$ This is a connected oriented surface with one boundary component for $n$ odd and with two boundary components if $n$ is even. The genus of $\surf_n$ is $g = \frac{n-1}{2}$ for odd $n$ and $g = \frac{n-2}{2}$ for $n$ even.

Hence we define the space
$$
\totsp_n := \{(\P, z,y ) \in \conf_n  \times \disk \times \C| y^2 = (z-x_1)\cdots(z-x_n) \}.
$$
Notice that $E_n$ has a natural projection $\pi:E_n\to \conf_n$ that maps $(\P, y, z) \mapsto \P$. The fiber of $\pi$ is the surface $\surf_n$.

It is natural to consider the complement of the $n$-points set in the disk: $\disk \setminus \P$. We have that $H_1(\disk \setminus \P)$ has rank $n$. The surface $\surf_n$ is a double covering of $\disk$ ramified along $\P$, hence it is natural to identify $\P$ as a subset of $\surf_n$. 
We define $\ddiskP:= \surf_n \setminus \P$ as the double covering of $\disk \setminus \P$ induced by $\surf_n \to \disk$.
Notice that for $n$ odd $H_1(\ddiskP)$ has rank $2n-1$. 

There is a projection 
$\totsp_n \stackrel{p}{\longrightarrow} \conf_n \times \disk 
$ given by $p:(\P,z,y) \mapsto (\P, z)$.
Hence $\totsp_n$ is a double covering of $\conf_n \times \disk$ ramified along the space $\confm{n-1} := \{ (\P, z) \in \conf_n \times \disk | z \in \P\}$. This is the configuration space of $n-1$ unordered distinct points in $\disk$ with one additional distinct marked point. In particular the complement of
$\confm{n-1} \subset \conf_n \times \disk$ is $\confm{n}$, so the complement of
$p^{-1}(\confm{n-1})$ in $\totsp_n$ is a double covering of $\confm{n}$ that we call $\confmd{n}$.
The fundamental group of $\confm{n}$ is the Artin groups $\Gb{n}$. Moreover (see for example \cite{bri_73}) the space $\confm{n}$ is a $K(\Gbn,1)$.





\begin{rem}\label{rem:globalsection}
Notice that the covering $\surf_n \into \totsp_n \stackrel{\pi}{\to} \conf_n$ admits a global section (see Definition \ref{def:section}) and hence $H_*(\totsp_n) = H_*(\totsp_n,\conf_n) \oplus H_*(\conf_n)$ and $$H_i(\totsp_n, \conf_n) = H_{i-1}(\conf_n; H_1(\surf_n)).$$
\end{rem}

We recall that the $\Z$-module $H_1(\surf_n)$  is endowed with a symplectic form given by the cap product. Moreover the action of $\pi_1(\conf_n)$ on $H_1(\surf_n)$  associated to the covering $\pi$ preserves this form. This monodromy representation is induced by the embedding of the braid group  $\pi_1(\conf_n)$  in the mapping class group of the surface $\surf_n$. Such a monodromy representation maps the standard genyerators of the braid groups to Dehn twists and is called geometric monodromy (see \cite{per_van_92, waj_99}).
Hence we can consider $H_1(\surf_n)$ as a $\pi_1(\conf_n)= \Br_n$-representation; we write also $\sym{g} :=H_1(\surf_n)$, where $g =\frac{n-1}{2}$ for $n$ odd, and $g = \frac{n-2}{2}$ for $n$ even.

The braid group $\Br_n = \Ga{n-1}$ maps on the permutation group $\perm_n$ on $n$ letters. Hence the group $\Br_n $ has a natural representation on $\Z^n$ by permuting cohordinates. We write $\st_n$ for this representation of $\Br_n$.

\section{Homology of some Artin groups}\label{sec:homol_artin}

We collect here some of the results
concerning the homology of $\Ga{n}$ and $\Gb{n}$ with constant coefficients and with coefficients in abelian local systems.
We follow the notation used in \cite{calmar}. 

Given an element $x \in \{0,1\}^n$ we can write it as a list of $0$'s and $1$'s. We identify such an element $x$ with a string of $0$'s and $1$'s.

Recall the definition of the following $q$-analog and $q,t$-analog polynomials with integer coefficients:
$$ 
[0]_q :=1, \qquad [m]_q := 1 + q + \cdots + q^{m-1} = \frac{q^m-1}{q-1} \mbox{ for }m\geq 1,
$$
$$
[m]_q! := \prod_{i=1}^m [m]_q, \qquad [2m]_{q,t} := [m]_q (1+tq^{m-1}),
$$
$$
[2m]_{q,t}!! := \prod_{i=1}^m [2i]_{q,t}\ =\ [ m ]_q!
\prod_{i=0}^{m-1} (1+tq^i),
$$
$$
\qbin{m}{i}_{q}\!\!: = \frac{[m]_q!}{[i]_q!
	[m-i]_q!},
\qquad
\qbin{m}{i}_{q,t}'\!\!: = \frac{[2m]_{q,t}!!}{[2i]_{q,t}!! [m-i]_q!} \ =
\qbin{m}{i}_{q}\prod_{j=i}^{m-1}(1+tq^j).
$$
In the following we specialize $q = -1$ and we write $\qbin{m}{i}_{-1}'$ for $\qbin{m}{i}_{-1,t}'$.


The homology of the Artin group of type $\Ga{n}$ with constant coefficients over the module $M$ is computed by the following complex $(C_i(\Ga{n}, \md), \partial)$.
\begin{df}
$$
C_i(\Ga{n}, \md) := \bigoplus_{|x| = i} \md.x 
$$
where the boundary is defined by:
$$
\partial 1^l = \sum_{h=0}^l-1 (-1)^h \qbin{l+1}{h+1}_{-1} 1^{h}01^{l-h-1}
$$
and if $A$ and $B$ are two strings
$$
\partial A0B = (\partial A)0B + (-1)^{|A|} A0 \partial B.
$$
\end{df}

Assume that the group $\Gb{n}$ acts on the module $\md = \ring[t^\pmu]$ mapping the first standard generator to multiplication by $(-t)$ and all other generators to multiplication by $1$. Then the homology of $\Gb{n}$ with coefficients on $\md$ is computed by the following complex $(C_*(\Gb{n}, \md),\DB) $.
\begin{df}\label{def:complesso}
The complex $C_*(\Gb{n}, \md)$ is given by:
$$
C_i(\Gb{n}, \md) := \bigoplus_{|x| = i} \md.x 
$$
where $\md.x$ is a copy of the module $\md$ generated by an element $x$ and $x \in \{0,1\}^n$ is a list of length $n$ and $|x| :=  |\{i \in 1, \ldots, n \mid x_i = 1\}|$.

When we represent an element $x$ that generates $C_*(\Gb{n}, \md)$ as a string of $0$'s and $1$'s, we put a line over the first element of the string, since it plays a special role, different from that of the complex $C_*(\Ga{n}, \md)$.

The boundary $\DB x$ for $C_*(\Gb{n}, \md)$ is defined by linearity from the following relations:
$$
\DB \overline{0}A = \overline{0}\DA A,
$$
$$
\DB \; \overline{1}1^{l-1} = 
\qbin{l}{0}_{-1}' \overline{0}1^{l-1} + \sum_{h=1}^{l-1} (-1)^{h}\qbin{l}{h}_{-1}' 
\overline{1}1^{h-1}01^{l-h-1}
$$
and
$$
\DB A0B = (\DB A)0B + (-1)^{|A|}A0\DA B.
$$
\end{df}

Hence we have:
\begin{thm}[\cite{salvetti}]
$$
H_*(\Ga{n}; \md) = H_*(C_*(\Ga{n}, \md), \DA)
$$
$$
H_*(\Gb{n}; \md) = H_*(C_*(\Gb{n}, \md), \DB)
$$
\end{thm}

Let $\Br(n) = \Ga{n-1}$ be the classical Artin braid group in $n$ strands. We recall the description of the homology of these groups
according to the results of \cite{cohen, fuks, vain}. 
We shall adopt a notation coherent with \cite{dps} (see also \cite{cal06}) for the description of 
the algebraic complex and the generators. Let $\F$ be a field.
The direct sum of the homology of $\Br(n)$ for $n \in \N = \Z_{\geq 0}$ is considered as a bigraded ring $\oplus_{d,n} H_d(\Br(n), \F)$ 
where the product structure $$H_{d_1}(\Br(n_1), \F) \times H_{d_2}(\Br(n_2), \F) \to H_{d_1+d_2}(\Br(n_1+n_2), \F)$$
is induced by the map $\Br(n_1) \times \Br(n_2) \to \Br(n_1 + n_2)$ that juxtaposes braids (see \cite{cohen_braids, cal06}).

\subsection{Braid homology over $\Q$}\label{ss:braid_homol}

The homology of the braid group with rational coefficients has a very simple description:
$$
H_d(\Br(n), \Q)= \left( \Q[x_0, x_1]/(x_1^2)\right)_{\deg = n, \dim= d}
$$
where $\deg x_i = i+1 $ and $\dim x_i = i$.
In the Salvetti complex for the classical braid group (see \cite{salvetti, dps} ) the element $x_0$ is represented by the string 0 and $x_1$ is 
represented by the string 10. In the representation of a monomial $x_0^ax_{1}^b$ we drop the last 0.

For example the generator of $H_1(\Br(4),\Q)$ is the monomial $x_0^2x_1$ and we can also write it as a string in the form $001$ (instead of $0010$, dropping the last $0$).

We denote by $A(\Q)$ the module $\Q[x_0, x_1]/(x_1^2)[t^\pmu]$.

\subsection{Braid homology over $\F_2$}

With coefficients in $\F_2$ we have:

$$
H_d(\Br(n), \F_2)= \F_2[x_0, x_1, x_2, x_3, \ldots]_{\deg = n, \dim= d}
$$
where the generator $x_i, i \in \N,$ has degree $\deg x_i = 2^i$ and homological dimension 
$\dim x_i = 2^i-1$.

In the Salvetti complex the element $x_i$ is represented by a string of $2^i -1$ 1's 
followed by one 0. In the representation of a monomial $x_{i_1}\cdots x_{i_k}$ we drop the last 0.

We denote by $A(\F_2)$ the module $\F_2[x_0, x_1, x_2, x_3, \cdots][t^\pmu]$.

\subsection{Braid homology over $\F_p$, $p>2$}

With coefficients in $\F_p$, with $p$ an odd prime, we have:
$$
H_d(\Br(n), \F_p)= \left( \F_p[h, y_1, y_2, y_3, \ldots] \otimes \Lambda[x_0, x_1, x_2, x_3, \ldots] 
\right)_{\deg = n, \dim= d}
$$
where the second factor in the tensor product is the exterior algebra over the field $\F_p$ 
with generators $x_i, i \in\N$. 
The generator $h$ has degree $\deg h = 1$ and homological dimension $\dim h=0$.
The generator $y_i, i \in \N$ has degree $\deg y_i = 2p^i$ and homological dimension 
$\dim y_i = 2p^i-2$.
The generator $x_i, i \in \N$ has degree $\deg x_i = 2p^i$ and homological dimension 
$\dim x_i = 2p^i-1$.

In the Salvetti complex the element $h$ is represented by the string $0$, the element $x_i$ is 
represented by a string of $2p^i -1$ $1$'s followed by one $0$. %

We remark that the term $\DA(x_i)$ is divisible by $p$. In fact, with generic coefficients (see \cite{cal06}), the differential $\DA(x_i)$ is given by a sum of terms with coefficients all divisible by the cyclotomic polynomial $\varphi_{2p^i}(q)$. Specializing to the trivial local system, with integer coefficients we have that all terms are divisible by $\varphi_{2p^i}(-1) = p$.

The element $y_i$ is represented by the following term (the differential is computed over 
the integers and then, 
after dividing by $p$, we  consider the result modulo $p$):
$$
\frac{\DA(x_i)}{p}.
$$

\subsection{Homology of the Artin group of type $B$}
Now let us recall some results on the homology of the groups $\Gb{n}$ with coefficients on the module $\F[t^\pmu]$, where the first standard generator of  $\Gb{n}$ acts with multiplication by $(-t)$ and all the others generators acts with multiplication by $1$.

In order to describe the elements of $C_*(\Gb{n};\md)$ we write $z_i$ for $\overline{1}1^{i-1}0$. Hence, if $x$ is a generator of $C_*(\Gb{n};\md)$ of the form $x = \overline{1}1^{i-1}0y$, then we write $x = z_iy$.

From \cite[Sec.~4.2]{calmar} we have
\begin{prop} \label{prop:hrazionale}
Let  $\F$ be a field of characteristic $p=0$. The $\F[t]$-module $\oplus_{i,n} H_i(\Gb{n}; \F[t^{\pmu}])$ has a basis given by
$$
\frac{\DB (z_{2i+1}x_0^{j-1})}{1+t},
$$
$$
\frac{\DB (z_{2i+1}x_0^{j-1}x_1)}{1+t},
$$
and
$$
\frac{\DB (z_{2i+2})}{1-t^2},
$$
where the first and the second kind of generators have torsion of order $(1+t)$, while the third kind of generators have torsion of order $(1-t^2)$.
\end{prop}

From the description given in \cite[Thm.~4.5, Thm.~4.12]{calmar} we have the following results.

\begin{prop}\label{prop:homol_p2}
For $p=2$ the  $\F[t]$-module $\oplus_{i,n} H_i(\Gb{n}; \F[t^\pmu])$ has a basis given by
\begin{equation}\label{generators_21}
\frac{\DB(z_{2^{h+1}(2m+1)+1}x_{i_1} \cdots x_{i_k})}{(1+t)}.
\end{equation}
and
\begin{equation}\label{generators_022}
\frac{\DB(z_{2^{h+1}(2m+1)+2^i}x_{i_1} \cdots x_{i_k})}{(1-t^2)^{2^{i-1}}}.
\end{equation}
where $i \leq i_1 \leq \cdots i_k$, $i \leq h$ and the first kind of generators have
torsion of order $(1+t)$ while the second kind of generators have torsion or order $(1-t^2)^{2^{i-1}}$.\end{prop}

\begin{prop}\label{prop:homol_pdisp}
For $p >2$ or $p=0$, lef $\F$ be a field of characteristic $p$. For $n$ odd the homology $H_i(\Gb{n}; \F[t^\pmu])$ is an $\F[t]$-module with torsion of order $(1+t)$.
\end{prop}
The Proposition \ref{prop:homol_pdisp} above is a consequence of Proposition \ref{prop:hrazionale} and \cite[Thm.~4.12]{calmar}, since 
both results provide a description of the module $\oplus_{i,n} H_i(\Gb{n}; \F[t^\pmu])$ with generators with torsion of order $(1+t)$ or $(1-t^2)^k$ for suitable exponents $k$'s. One can verify that when $n$ is odd the torsion of the generators of degree $n$ is always $(1+t)$.

Next we compute the homology groups  and $H_*(\Gb{n}; \F_2[t]/(1-t^2))$ 
using the explicit description of \cite[Sec.~4.3 and 4.4]{calmar}.
As a special case of \cite[Prop.~4.7]{calmar} we have 
the isomorphism
\begin{equation}\label{eq:decomposizione}
H_i(\Gb{n}; \F_2[t]/(1-t^2)) = h_{i}(n,2) \oplus h_{i}'(n,2)
\end{equation}
where the two summands are determined by the following exact sequence:
$$
0 \to h_{i}'(n,2) \to H_i(\Gb{n}; \F_2[t^{\pmu}]) \stackrel{(1+t^2)}{\longrightarrow} H_i(\Gb{n}; \F_2[t^{\pmu}]) \to h_{i}(n,2) \to 0.
$$
For odd $n$ all the elements of $H_i(\Gb{n}; \F_2[t^{\pmu}])$ are multiple of $x_0$ and hence have $(1+t)$-torsion (see Proposition \ref{prop:homol_p2}). Hence the multiplication by 
$(1+t^2)$ is the zero map and the generators of  $h_{i}'(n,2)$ and $h_{i}(n,2)$ are  in bijection with a set of generators of $H_i(\Gb{n}; \F_2[t^{\pmu}])$.

In particular, from a direct computation (see \cite[\S 4.4]{calmar}) we have:
\begin{prop} \label{prop:generators}
For odd $n$ the homology $H_*(\Gb{n}; \F_2[t]/(1-t^2))$ is generated, as an $\F_2[t]$-module, by the classes of the form 
\begin{equation}\label{generators1}
\gp (z_c, x_0 x_{i_1} \cdots x_{i_k}) := (1-t) z_{c+1} x_{i_1} \cdots x_{i_k}.
\end{equation}
that correspond to the generators of $h_{i}'(n,2)$
and
\begin{equation}\label{generators2}
\gs(z_c, x_0 x_{i_1} \cdots x_{i_k}):=\frac{\DB(z_{c+1}x_{i_1} \cdots x_{i_k})}{(1+t)}.
\end{equation}
that correspond to generators of $h_{i}(n,2)$. Here $0 \leq i_1 \leq \cdots i_k$, $c$ is even and both kind of generators have
torsion $(1+t)$.
\end{prop}

Here we do not provide a description of the generators of $H_i(\Gb{n}; \F[t^\pmu])$ for a field $\F$ of characteristic $p>2$ and we also avoid a detailed presentation of a set of generators of  $H_i(\Gb{n}; \F[t]/(1+t))$ and $H_i(\Gb{n}; \F[t]/(1-t^2))$ for a generic field $\F$. For such a description, we refer to \cite{calmar}.

Since it will be useful in Section \ref{s:4tor} and in particular in the proof of Lemma \ref{lem:tau_torsion}, we provide sets of elements $\cB'$, $\cB''$ of the $\Z$-modules $H_i(\confm{n};\Z)\simeq H_i(\Gb{n}; \Z[t]/(1+t))$ and $H_i(\confmd{n};\Z)\simeq H_i(\Gb{n}; \Z[t]/(1-t^2))$ for $n$ odd, such that the following two condition are satisfied:
\begin{enumerate}[label=(\roman*)]
	\item $\cB'$ (resp.~$\cB''$) induces a base of the homology of $H_i(\confm{n};\Q) $ (resp.~$H_i(\confmd{n};\Q)$);
	\item the images of the elements of $\cB'$ (resp.~$\cB''$) in $H_i(\confm{n};\Z_p)$ (resp.~$H_i(\confmd{n};\Z_p)$) are
	linearly independent for any prime $p$.
\end{enumerate}
\begin{df}\label{def:BB}
Let $n$ be an odd integer. We define the sets $\cB' \subset H_i(\confm{n};\Z)$ (for $e=1$) and $\cB'' \subset H_i(\confmd{n};\Z)$ (for $e=2$) given by the following elements:
$$ 
\omega_{2i,j,0}^{(e)}:= \frac{\DB(z_{2i+1}x_0^{j-1})}{(1+t)} \mbox{ and } \widetilde{\omega}_{2i,j,0}^{(e)}:= \frac{(1-(-t)^e)z_{2i+1}x_0^{j-1}}{(1+t)} \qquad \mbox{ for }j>0;
$$
and
$$
\omega_{2i,j,1}^{(e)}:=\frac{\DB(z_{2i+1}x_0^{j-1}x_1)}{(1+t)} \mbox{ and } \widetilde{\omega}_{2i,j,1}^{(e)}:=\frac{(1-(-t)^e)z_{2i+1}x_0^{j-1}x_1}{(1+t)} \qquad \mbox{ for }j>0.
$$
\end{df}
It follows from \cite[\S4.2]{calmar}
that the elements above provide a basis for $H_*(\confm{n};\Q) $ (resp.~$H_*(\confmd{n};\Q)$) for $n$ odd (condition (i)). 

For condition (ii), one can check that the elements in $\cB'$ and $\cB''$ define, mod $2$, a subset of the bases of  $H_i(\Gb{n}; \Z_2[t]/(1+t))$ and $H_i(\Gb{n}; \Z_2[t]/(1-t^2))$ given in \cite[\S4.4]{calmar} and, mod $p$ for an odd prime, a subset of the bases of $H_i(\Gb{n}; \Z_p[t]/(1+t))$ and $H_i(\Gb{n}; \Z_p[t]/(1-t^2))$  given in \cite[\S 4.6]{calmar}.

Hence the following claim follows:
\begin{prop}\label{prop:BB}
For $n$ odd the elements of $\cB'$ (resp.~$\cB''$) are a free set of generator of a maximal free $\Z$-submodule of $H_i(\confm{n};\Z)$ (resp.~$H_i(\confmd{n};\Z)$).
\end{prop}

\begin{prop}\label{prop:hrazionale_minus_t}
Let $n$ be an even integer. Let $\F$ be a field of characteristic $0$.
The Poincar\'e polynomial of 
$H_*(\Gb{n}; \F[t]/(1-t))$ is $(1+q)q^{n-1}
$
and a base of the homology is given by the following generators
$$\frac{\DB (z_{n})}{1-t},  z_{n}.$$
\end{prop}	
\begin{proof}
This follows from Proposition \ref{prop:hrazionale} and by studying the long exact sequence associated to $$
0 \to \F[t^\pmu] \stackrel{(1-t)}{\longrightarrow} \F[t^\pmu] \to \F[t]/(1-t) \to 0
$$ as in \cite[\S~4.2]{calmar}.
\end{proof}

\section{Exact sequences} \label{sec:exact_seq}



We recall that the double covering $\pi: \confmd{n} \to \conf_n$ has a continuous section $s$  (see also \cite[p. 30]{bianchi}) that we can define as follows.
\begin{df} \label{def:section}
Given a monic polynomial $p$ with $n$ distinct roots $x_1, \ldots, x_n$ such that $|x_i| <1$ for all $i$ we can map
$$s: p \mapsto \left(p, z:= \frac{(\max_i |x_i|) +1}{2}, \sqrt{p(z)}) \right)$$
where we choose $\sqrt{p(z)}$ as a continuous function as follows: since $p(z) = \prod_i(z-x_i)$ is a product of complex numbers with $\Re (z-x_i) >0$, we can choose $\sqrt{z-x_i}$ to be the unique square root with $\Re \sqrt{z-x_i} >0$ and hence we can define $\sqrt{p(z)} := \prod_i \sqrt{z-x_i}$.
\end{df}
Clearly the section $s:\conf_n \to \confmd{n}$ is a continuous lifting of the section  $\overline{s}:\conf_n \to \confm{n}$ that maps $p \mapsto (p, z)$.
The homology map $s_*$ is injective and the short exact sequence
$$
0 \to H_i(\conf_n) \stackrel{s_*}{\longrightarrow} H_i(\confmd{n}) \stackrel{J}{\to} H_i(\confmd{n}, \conf_n) \to 0
$$
is split. Hence $H_i(\confmd{n}) \simeq H_i(\conf_n) \oplus H_i(\confmd{n}, \conf_n)$.
Moreover the same argument applied to the section $\overline{s}$ for $\pi: \confm{n} \to \conf_n$ implies that we have the splitting $H_i(\confm{n}) \simeq H_i(\conf_n) \oplus H_i(\confm{n}, \conf_n)$.	

\begin{rem}\label{rem:Cn_sections}
Let $\sigma: \confmd{n} \to \confmd{n}$ the only nontrivial automorphism of the double covering $\confmd{n} \to \confm{n}$.
We can obviously define another section  $s': \conf_n \to \confmd{n}$ such that $\sigma s = s'$, $\sigma s' = s$.
Hence we can include $S:\conf_n \times \{1, -1\} \into \confmd{n}$ and the restriction of $\sigma$ exchanges the components $\conf_n \times \{1\}$ and $\conf_n \times \{-1\}$. Hence we can understand the inclusion $s:\conf_n \into \confmd{n}$ in homology via the following diagram:
$$
H_i(\conf_n) \stackrel{\Id \otimes 1}{\longrightarrow} H_i(\conf_n) \otimes \ring[t]/(1-t^2) \stackrel{S_*}{\to} H_i(\confmd{n}) 
$$
and the composition is injective.
\end{rem}

\begin{rem} \label{rem:s_sprime}
If $n$ is odd the two sections $s,s'$ are homotopic. In fact we can define in a unique way a continuous family of maps
$
s_t:\conf_n \to \confmd{n}
$
such that $s_0 = s$ and
$$
s_t: \mapsto \left(p, z_t:= e^{2\pi i t}\frac{(\max_i |x_i|) +1}{2}, \sqrt{p(z_t)} \right)
$$
Since $p(z_t)$ is a product of $n$ factors it is clear that  for $n$ odd we have $s_1 = s'$.
As a consequence $s_* = s'_*$.\end{rem}

According to Bianchi (\cite{bianchi}) we consider the following decomposition.
The space $\totsp_n$ is the union of the double covering $\confmd{n}$ and a subset diffeomorphic to $\confm{n-1}$. 
Let $N$ be a small tubolar neighborhood of $\confm{n-1}$ in $\totsp_n$ and let $M$ be the closure of its complement in $\totsp_n$. The complement $M$ is homotopy equivalent to $\confmd{n}$, while $\partial N = M \cap N$ is diffeomorphic to $\confm{n-1} \times S^1$. Moreover $M$ contains a subspace $\conf_n = s(\conf_n) \subset \totsp_n $.

Hence there is a relative Mayer-Vietoris long exact sequence 
$$
\cdots \to H_i(\partial N) \to H_i (N) \oplus H_i(M, \conf_n) \to H_i(\totsp_n, \conf_n) \to \cdots
$$
that is equivalent to the following long exact sequence:
$$
\cdots \to H_i(\confm{n-1}) \oplus H_{i-1}(\confm{n-1}) \otimes H_1(S^1)\stackrel{\iota}{\to} H_i (\confm{n-1}) \oplus H_i(\confmd{n}, \conf_n) \to H_i(\totsp_n, \conf_n) \to \cdots
$$
Notice that from Kunneth decomposition the restriction of the map $\iota$ induces an isomorphism between the terms (again, see also \cite[Lem.~58]{bianchi}):
$$
\iota: H_i(\confm{n-1}) \to H_i(\confm{n-1})
$$
and the restriction of $\iota$ to the second summand $H_{i-1}(\confm{n-1}) \otimes H_1(S^1)$ maps to zero if we project to the term $H_i(\confm{n-1})$.
Hence we can simplify our exact sequence as follows:
\begin{equation}\label{eq:bia1}
\cdots \to  H_{i-1}(\confm{n-1}) \otimes H_1(S^1)\stackrel{\iota}{\to} H_i(\confmd{n}, \conf_n) \to H_i(\totsp_n, \conf_n) \to \cdots
\end{equation}


We already know that $\confm{n}$ is a classifying space for the Artin group of type $B_n$, and hence we have:
$$
H_i(\confm{n-1}) = H_i(\Gb{n-1}).
$$
Moreover the forgetful map $\confm{n-1} \to \conf_n$ is a covering with generic fibre given by the discrete set $P \in \conf_n$ and hence induces the isomorphism (see \cite[Lem.~8]{bianchi}):
$$
H_i(\Br_{n}; \st_n) = H_i(\conf_{n}; \st_n) \simeq H_i(\confm{n-1}).
$$
Recall that $\partial N = S^1 \times \confm{n-1}$. Hence with respect to the fibration $\partial N \to \conf_n$, with fiber $S^1 \times P$, the monodromy action of $\Br_n = \pi_1(\conf_n)$ on $H_1(S^1 \times P)$ is exactly the permutation action. It follows that we have isomorphic $\Br_n$-representations:
 $H_1(S^1 \times P) \simeq \st_n$.

As we already noticed in Remark \ref{rem:globalsection}:
$$
H_i (\totsp_n, \conf_n) \simeq H_{i-1}(\Br_n; H_1(\surf_n))
$$

Finally the term $ H_i(\confmd{n}, \conf_n)$ is isomorphic to $H_{i-1}(\Br_n; H_1(\ddiskP))$.

Then we can rewrite \eqref{eq:bia1} as follows:
\begin{equation}\label{eq:bia2}
\cdots \to H_{i-1}(\Br_n; H_1(S^1 \times P)) \stackrel{\iota}{\to} H_{i-1}(\Br_n; H_1(\ddiskP)) \to  H_{i-1}(\Br_n; H_1(\surf_n)) \to \cdots
\end{equation}

Actually we can see that \eqref{eq:bia2} is the long exact sequence for the homology of the group $\Br_n$ associated to the short exact sequence of coefficients
$$
0 \to H_1(S^1 \times P) \to H_1(\disk \times P) \oplus H_1(\ddiskP) \to H_1(\surf_n) \to 0
$$
that follows from the Mayer-Vietoris long exact sequence associated to the decomposition: $\surf_n = \disk \times P \cup \ddiskP$, with, up to homotopy, $\disk \times P \cap \ddiskP \simeq S^1 \times P$.

\begin{df} \label{df:mu}
We write $p= (p_1, \{p_2, \ldots, p_n\})$ for a point in $\confm{n-1}$ and   $\overline{p} := \{p_1, \ldots, p_n\}$. Let $$\delta(\overline{p}):= \frac{1}{2}\min \left( \{ |p_1 - p_i|,  2 \leq i  \leq n\} \cup \{ 1-|p_1|\}
\right).$$
We define the map $\mu$ is given  by 
$$
\confm{n-1} \times S^1 \ni (p,e^{it}) \mapsto (p_1 + \delta(\overline{p})e^{it}, \overline{p}) \in \confm{n}
$$
\end{df}

In order to understand $H_{i-1}(\Br_n; H_1(S^1 \times P)) \stackrel{\iota}{\to} H_{i-1}(\Br_n; H_1(\ddiskP))$ we can consider the commuting diagram
\begin{equation}\label{diag:trerighe0}
\begin{tabular}{c}
\xymatrix @R=2pc @C=2pc {
H_{i-1}(\Br_n; H_1(S^1 \times P)) \ar[r]^\iota \ar[d]^\simeq & H_{i-1}(\Br_n; H_1(\ddiskP))\ar[d]^\simeq  \\
H_{i-1}(\confm{n-1}) \otimes H_1(S^1) \ar[d]^{\mu_*} \ar[r] & H_i(\confmd{n}, \conf_n) \\
H_i(\confm{n}) 
&H_i(\confmd{n})  \ar[u]_J
}
\end{tabular}
\end{equation}
where $\mu_*$ is induced by the map $\mu: \confm{n-1} \times S^1 \to \confm{n}$
and $J$ is induced by the inclusion $\confmd{n} \to (\confmd{n}, \conf_n)$.


In \cite[Thm.~12]{bianchi} Bianchi shows that $H_*(\confm{n}) = H_*(\conf_n) \oplus H_{*-1}(\confm{n-1})$, where the first summand is the image of the map induced by the natural inclusion $\overline{s}: \conf_n \to \confm{n}$
and the projection to the first summand corresponds to the map induced by the forgetful map $r:\confm{n} \to \conf_n$. This argument is also implicit in \cite[Ch. 1,~\S 5]{vas} since Vassiliev shows that $H_i(\conf_n; H_1(\disk \setminus \P)) \simeq H_i(\confm{n-1})$ and from the spectral sequence associated to the projection $r:\confm{n} \to \conf_n$ (with section $\bar s$) we get $H_i(\confm{n}, \conf_n) =  H_{i-1}(\conf_n; H_1(\disk \setminus \P))
$. 

In order to provide an explicit description of the homology homomorphism induced by $\mu$, we give another proof of this splitting.
We have the short exact sequence:
$$
0 \to H_i(\conf_n) \stackrel{\overline{s}_*}{\longrightarrow} H_i(\confm{n}) \to H_i(\confm{n}, \conf_n) \to 0
$$
and since $\overline{s}$ is a section of $r$ we have that the exact sequence splits and 
$
H_i(\confm{n}, \conf_n) = \ker [r_*: H_i(\confm{n}) \to H_i(\conf_n)].
$

Let now define the decomposition 
$\conf_n \times \disk = \confm{n} \cup \confm{n-1} \times \disk$, where we naturally identify $\confm{n}$ with a subset of $\conf_n \times \disk$ mapping
$$
p= (p_1, \{p_2, \ldots, p_{n+1}\}) \mapsto (\{p_2, \ldots, p_{n+1}\}, p_1)
$$
and where we identify $\confm{n-1}\times \disk$ with a subset of $\conf_n \times \disk$ mapping
$$
(p',q)= ((p'_1, \{p'_2, \ldots, p'_{n}\}),q) \mapsto (\{p'_1, \ldots, p'_n\}, p_1+\delta(\overline{p}')q).
$$
Clearly $\confm{n} \cap \confm{n-1} \times \disk \simeq \confm{n-1} \times S^1$.
Hence we get the associated Mayer-Vietoris exact sequence:
$$
\cdots \to H_i(\confm{n-1} \times S^1) \to H_i(\confm{n}) \oplus H_i(\confm{n-1}) \to H_i(\conf_n) \to \cdots.
$$
Since $H_i(\confm{n})$ decomposes as $H_i(\confm{n}) = H_i(\confm{n},\conf_n) \oplus H_i(\conf_n)$ and the map $H_i(\confm{n}) \to H_i(\conf_n)$ is surjective, with kernel given by $H_i(\confm{n}, \conf_n)$, we have the isomorphism
$$
H_i(\confm{n-1} \times S^1) \simeq  H_i(\confm{n},\conf_n) \oplus H_i(\confm{n-1}).
$$
Finally notice that $H_i(\confm{n-1} \times S^1) = H_i(\confm{n-1}) \oplus  H_{i-1}(\confm{n-1}) \otimes H_1(S^1)$ and $H_{i-1}(\confm{n-1}) \otimes H_1(S^1)$ has trivial projection onto $H_i(\confm{n-1})$ since it factors through 
$$
H_{i-1}(\confm{n-1}) \otimes H_1(S^1) \to H_{i-1}(\confm{n-1}) \otimes H_1(\disk) \to H_i(\confm{n-1}).
$$
Recalling the definition of $\mu$ we obtain:
\begin{prop} \label{prop:isodellesemplicifazioni}The following groups are isomorphic
$$
H_{i-1}(\confm{n-1}) \simeq H_i(\confm{n},\conf_n)
$$
and the isomorphism is induced by the map $\mu$.
\end{prop}

The increasing filtration $$\Filt^i:= \langle  A \mid A \mbox{ is a string that contains at least one } 0 \mbox{ among the first } i \mbox{ entries.}\rangle$$ of the complex $C_*(\Gb{n}, \md)$ introduced in Definition \ref{def:complesso} 
induces a spectral sequence 
\begin{equation}\label{eq:ss_bn}
E_{ij}^2 = H_j(\Br(n-i); \md) \Rightarrow H_{i+j} (\Gb{n}; \md).
\end{equation}
\begin{rem}\label{rem:collapsss}
If the action of $\Gb{n}$ on the module $\md$ is trivial the spectral sequence collapses at $E^2$. This fact can be proved with a quite technical argument: in fact one can see that all the non-zero differentials of the spectral sequence are divided by a coefficient $(1+t)$, where $-t$ correspond to the action of the first standard generator of the group $\Gb{n}$ on the module $\md$. In particular, if the action is trivial, all the differentials of the spectral sequence are trivial (see \cite{calmar} for a detailed analysis of this spectral sequence). Another more elementary argument is the following: the splitting $H_*(\confm{n}) \simeq H_*(\conf_n) \oplus H_{*-1}(\confm{n-1})$ induces the decomposition (see also \cite{gorj}) $$
H_*(\confm{n}) = \oplus_{i \geq 0} H_{*-i} (\conf_{n-i})
$$ that is isomorphic to the $E^2$-term of the spectral sequence above
and the same argument proves the splitting for any system of coefficients where the group $\Gb{n}$ acts trivially.
\end{rem}
As a consequence of the previous remark the $E^\infty$ term of the spectral sequence given in \eqref{eq:ss_bn} is isomorphic to the homology $H_*(\Gb{n}; \md)$.
Moreover the maps $\overline{s}, r$ induce respectively the inclusion of the first column and the projection onto the first column of the spectral sequence.

Hence from the previous description of the spectral sequence and from Proposition \ref{prop:isodellesemplicifazioni} we have the following result.
\begin{prop} \label{prop:image_mu}
The image of $\mu_*$ corresponds to the direct sum of all the columns but the first in the spectral sequence \eqref{eq:ss_bn} for $H_*(\confm{n})$.
\end{prop} 
\begin{rem}\label{rem:Q_mu}
Let $\F$ be a field of characteristic $0$. Then for $i>1$ the map
$$
\mu_*: H_{i-1}(\confm{n-1}) \otimes H_1(S^1) \to H_i(\confm{n})
$$
is an isomorphism. 
\end{rem}
%
\begin{prop}\label{prop:tau1}
Let $\tau: H_i(\confm{n}) \to H_i(\confmd{n})$ be the transfer map induced by the double covering $\confmd{n} \to \confm{n}$.
The following diagram commutes:
\begin{equation}\label{diag:trerighe}
\begin{tabular}{c}
\xymatrix @R=2pc @C=2pc {
	H_{i-1}(\Br_n; H_1(S^1 \times P)) \ar[r]^\iota \ar[d]^\simeq & H_{i-1}(\Br_n; H_1(\ddiskP))\ar[d]^\simeq  \\
	H_{i-1}(\confm{n-1}) \otimes H_1(S^1) \ar[d]^{\mu_*} \ar[r] & H_i(\confmd{n}, \conf_n) \\
	H_i(\confm{n}) \ar[r]^{\tau} &H_i(\confmd{n})  \ar[u]_J
}
\end{tabular}
\end{equation}
\end{prop}
\begin{proof}

Let $\sigma_k$ be a $k$-simplex in $\confm{n-1}$. Then $p = \sigma_k(q)$ is a point in $\confm{n-1}$ and we write $p= (p_1, \{p_2, \ldots, p_n\})$. Let 
$$
\Sigma_p = \{(x,y) \in \C^2 \mid y^2 = (x-p_1)\cdots(x-p_n)\}
$$ be the double covering of $\C$ ramified around $\overline{p} = \{p_1, \ldots, p_n\}$. Then if we consider the projection $\pi: \totsp_n \to \conf_n$ we have $\Sigma_p = \pi^{-1} (\overline{p})$.

Let $\epsilon$ be the automorphism of $\totsp_n$ that maps $( \overline{p}, x,y) \mapsto (\overline{p},x,-y)$. It is clear that the restriction of $\epsilon$ to $\confmd{n}$ is the automorphism $\sigma: \confmd{n} \to \confmd{n}$.

Let $D_p$ be the intersection $N \cap \Sigma_p$. We can assume that $D_p$ is diffeomorphic to a closed disk and the restriction of the projection $\pi_x: (x, y) \mapsto x$ to $D_p$ is a double covering of a small disk around $p_1$ in $\C \setminus \{p_2, \ldots, p_n\}$ ramified in $p_1$.

Let $\overline{\sigma}_k$ be the projection of $\sigma_k$ to $\conf_n$. The restriction of the tubolar neighborhood $N$ to $\overline{\sigma}_k$ is a trivial bundle. So we can define a parametrization $\gamma_p:[0,1]\to \partial D_p$ that is continuous in $p$. Then $\gamma_p$ represents a generator of $H_1(\partial D_p)$.

Let $\zeta$ be the standard generator of $H_1(S^1)$.
Hence the map $H_{i-1}(\confm{n-1}) \otimes H_1(S^1) \to H_i(\confmd{n}, \conf_n) $ is induced by mapping 
$\sigma_k \otimes \zeta \mapsto  \sigma_k \times \gamma $ defined as  $(\sigma_k \times \gamma) (q, t) = ( \overline{\sigma}_k(q), \gamma_{\sigma_k(q)}(t))$.

We can replace $\gamma_p$ by $\gamma'_p + \gamma''_p$ where we define $\gamma'_p (t) := \gamma_p(t/2)$ and $\gamma''_p(t) := \gamma_p((1+t)/2)$, both for $t \in [0,1]$. 

Hence we have that the map $H_{i-1}(\confm{n-1}) \otimes H_1(S^1) \to H_i(\confmd{n}, \conf_n) $ is induced by mapping 
$\sigma_k\otimes \zeta \mapsto  \sigma_k \times (\gamma' + \gamma'')$ defined as above. Moreover $\epsilon(\sigma_k \times \gamma') = \sigma_k \times \gamma''$ and $\epsilon(\sigma_k \times \gamma'') = \sigma_k \times \gamma'$.

Recall that $\mu(p,e^{it}) =  (p_1 + \delta(\overline{p})e^{it}, \overline{p})$. Hence, up to a suitable choice of  the tubular neighborhood $N$ and of the parametrization $\gamma$, we can assume that
$$
p_1 + \delta(\overline{p})e^{it} = \pi_x(\gamma'_{p}(t)).
$$
This implies that $\mu_*(\sigma_k\otimes \zeta) = \sigma_k \times \pi_x(\gamma')= \sigma_k \times \pi_x(\gamma'')$, where we define $ (\sigma_k \times \pi_x(\gamma'))(q,t) = (p_1 + \delta(\overline{p})e^{it},\overline{p})$.

It is now clear that $\sigma_k \times \gamma'$ and $\sigma_k \times \gamma''$ are both liftings of $\mu_*(\sigma_k\otimes \zeta)$ and since $\epsilon$ exchanges the two liftings, we have that 
the map $\tau: H_i(\confm{n}) \to H_i(\confmd{n})$ is the transfer map induced by the double covering $\confmd{n} \to \confm{n}$.
\end{proof}
In the case of $n$ odd a different proof of Proposition \ref{prop:tau1} can be found in \cite[Lem.~58]{bianchi}.

\begin{prop}\label{prop:commut_tau}
The following diagram commutes:
$$
\begin{tabular}{c}
\xymatrix @R=2pc @C=2pc {
H_i(\confm{n})\ar[d]^\simeq \ar[r]^{\tau} &H_i(\confmd{n}) \ar[d]^\simeq \\
H_i(\Gb{n}; \ring[t]/(1+t)) \ar[r]^{1-t} & H_i(\Gb{n}; \ring[t]/(1-t^2))	.
}
\end{tabular}
$$
where in the bottom row we are considering the map induced by the $(1-t)$-multiplication map $C_*(\Gb{n}, \ring[t]/(1+t)) \to C_*(\Gb{n}, \ring[t]/(1-t^2))$ 
\end{prop}
\begin{proof}

The complex $C_*(\Gb{n}, \ring[\Z]) = C_*(\Gb{n}, \ring[t^\pmu])$ computes the homology of the infinite cyclic cover $\confmd{n}^\Z$ associated to the homomorphism $\Gb{n} \to \Z$ that maps the first standard generator to multiplication by $(-t)$ and all other generators to multiplication by $1$.

Hence the complex $C_*(\Gb{n}, \ring[\Z_2]) = C_*(\Gb{n}, \ring[t]/(1-t^2))$ computes the homology of the double cover $\confmd{n}$ and $C_*(\Gb{n}, R) = C_*(\Gb{n}, \ring[t]/(1+t))$ computes the homology of $\confm{n}$.

Since the non-trivial monodromy associated to the double cover $\confmd{n} \to \confm{n}$ is induced by the first generator of $\Gb{n}$, the transfer of a cycle  in $C_*(\Gb{n}, \ring[t]/(1+t))$ to a cycle in $C_*(\Gb{n}, \ring[t]/(1-t^2))$ is given by the multiplication by $(1 -t)$.
\end{proof}


\begin{rem} \label{rem:J}
Recall the isomorphism $H_i(\confmd{n}) \simeq H_i(\Gb{n}; \ring[t]/(1-t^2))$. Let $\ring$ be a field of characteristic $p$. For $p \neq 2$ the second term decomposes as $H_i(\Gb{n}; \ring[t]/(1+t)) \oplus H_i(\Gb{n}; \ring[t]/(1-t)) $ and moreover for $n$ odd  the term $H_i(\Gb{n}; \ring[t]/(1-t))$ is trivial. Again, this follows from the fact that the module $H_i(\Gb{n}; \ring[t^\pmu])$ has $(1+t)$-torsion and from the homology long exact sequence associated to
	$$
	0 \to \ring[t^\pmu] \stackrel{1-t}{\longrightarrow} \ring[t^\pmu ] \to \ring[t]/(1-t) \to 0.
	$$
	In particular the homology groups $H_i(\confmd{n}) \simeq H_i(\Gb{n}; \ring[t]/(1-t^2))$ and $H_i(\confm{n}) \simeq H_i(\Gb{n}; \ring[t]/(1+t))$ are isomorphic and the isomorphis is induced by the quotient map $\ring[t]/(1-t^2) \to \ring[t]/(1+t)$.
	Hence 
	we can consider the commuting diagram
	$$
	\begin{tabular}{c}
	\xymatrix @R=2pc @C=2pc {
		0 \ar[r]& H_i(\conf_n) \ar[d]^\simeq \ar[r] & H_i(\confmd{n})	\ar[d]^\simeq \ar[r]^J & H_i(\confmd{n}, \conf_n) \ar[d]^\simeq  \ar[r] & 0\\
		0 \ar[r] & H_i(\conf_n) \ar[r] & H_i(\confm{n}) \ar[r]^{\overline{J}} & H_i(\confm{n}, \conf_n) \ar[r] & 0.
	}
	\end{tabular}
	$$
	where the last vertical map is an isomorphism from the five lemma. 
	From the right square we have that the map $J$ corresponds to the homomorphism
	$$
\overline{J}:	H_i(\Gb{n}; \ring[t]/(1+t)) \to H_i((\Gb{n}, \Ga{n-1}); \ring[t]/(1+t))
	$$
	associated to the inclusion $\Ga{n-1} \into \Gb{n}$ induced by $\conf_n \into \confm{n}$.  From the short exaxt sequence in the second row of the diagram above we have that the homomorphism $$\overline{J}: H_i(\confm{n}) = H_i(\confm{n}, \conf_n) \oplus H_i(\conf_n) \to H_i(\confm{n}, \conf_n)$$ is the projection on the first term of the direct sum.
\end{rem}

\begin{lem}\label{lem:tau_invertible}
If $p$ is an odd prime or $p=0$ and $\ring$ is a field of characteristic $p$, then the homomorphism
$$
\overline{\tau}:H_i(\Gb{n}; \ring[t]/(1+t)) \stackrel{1-t}{\longrightarrow} H_i(\Gb{n}; \ring[t]/(1-t^2))
$$
is invertible for $n$ odd.
\end{lem}
\begin{proof}
This follows since, for odd $n$, the homology group $H_i(\Gb{n}; \ring[t^\pmu])$ has $(1+t)$-torsion (see Proposition \ref{prop:homol_pdisp}
). 
For $p \neq 2$ we have that  $H_i(\Gb{n}; \ring[t]/(1-t)) = 0$ (see Remark \ref{rem:J}) and hence 
$$
H_i(\Gb{n}; \ring[t]/(1-t^2)) \simeq H_i(\Gb{n}; \ring[t]/(1+t))
$$
and the map $\overline{\tau}$ is equivalent to the multiplication map
$$
H_i(\Gb{n}; \ring[t]/(1+t)) \stackrel{1-t}{\longrightarrow} H_i(\Gb{n}; \ring[t]/(1+t))
$$
and $(1-t)$ is invertible.
\end{proof}
\begin{rem} \label{rem:Q_tau_J}
The decompositions $H_i(\confmd{n}) \simeq H_i(\Gb{n}; \ring[t]/(1-t^2)) = H_i(\Gb{n}; \ring[t]/(1+t)) \oplus H_i(\Gb{n}; \ring[t]/(1-t))$ and $H_i(\confm{n}) = H_i(\confm{n}, \conf_n) \oplus H_i(\conf_n)$ give the following consequences for $n$ odd and $\F$ a field of characteristic $0$. Since $H_i(\Gb{n}; \ring[t]/(1-t))$ has Poincar\'e polynomial $(1+q)q^{n-1}$ (Proposition \ref{prop:hrazionale_minus_t}) and since $H_*(\conf_n)$ has Poincar\`e polynomial $(1+q)$ (see \S~\ref{ss:braid_homol}), we have that the map
$$
J: H_i(\confmd{n})  \to H_i(\confmd{n}, \conf_n) 
$$
is an isomorphism for $i>1$.
Moreover the argument of Lemma \ref{lem:tau_invertible} implies that
the map
$$\tau: H_*(\confm{n}) \to H_*(\confmd{n})$$ is injective and its cokernel has Poincar\'e polynomial $(1+q)q^{n-1}$.
\end{rem}

\begin{prop} \label{prop:coker_n_pari}
Let $n$ be even and $\F$ a field of characteristic $0$. Then for $i>1$ the map 
$$
\iota: H_{i-1}(\Br_n; H_1(S^1 \times P))  \to H_{i-1}(\Br_n; H_1(\ddiskP))
$$
is injective and its cokernel has rank $1$ for $i = n-1, n-2$ and $0$ otherwise.
\end{prop}
\begin{proof}
The result follows from Remark \ref{rem:Q_mu} and Remark \ref{rem:Q_tau_J}.
\end{proof}
\begin{thm} \label{teo:tau_restricted}
Consider the decomposition $H_i(\confm{n}) = H_i(\confm{n}, \conf_n) \oplus H_i(\conf_n)$ associated to the inclusion $\overline{s}:\conf_n \into \confm{n}$ and $H_i(\confmd{n}) = H_i(\confmd{n}, \conf_n) \oplus H_i(\conf_n)$ associated to the inclusion $s:\conf_n \into \confmd{n}$. If $n$ is odd the following inclusion holds: $$\tau(H_*(\conf_n)) \subset H_*(\conf_n)$$
and for $x \in H_*(\conf_n)$ we have that $\tau(x) = 2x$.
\end{thm}

\begin{proof}
Since $\tau$ is the transfer map, we can consider the following diagram:
	$$
\begin{tabular}{c}
\xymatrix @R=2pc @C=2pc {
s\conf_n  \sqcup s' \conf_n \ar[d] \ar[r] & \conf_n \ar[d]^{\overline{s}} \\
\confmd{n} \ar[r]& \confm{n}
}
\end{tabular}
$$
where the left vertical map is the natural inclusion and the horizontal maps are $2:1$ projections that induces the transfer map $\tau$ and its restriction to $H_*(\conf_n)$. 
Then we have the following commuting diagram in homology:
$$
\xymatrix{ 
H_*(\conf_n) 
\ar[d]^{\overline{s}_*} \ar[drr]^(0.65){s_*+s'_*} \ar[rr]_(0.44){\tau_{|\overline{s}_*H_*(\conf_n)}} 
&& 
H_*(s(\conf_{n})  \sqcup s'(\conf_{n}))
\ar[d]^{i_*} 
\\
H_*(\confm{n}) 
\ar[rr]^\tau 
&& 
H_*(\confmd{n}) 
 }
$$
Hence given a cycle $x \in H_*(\conf_n)$ we have that 
$\tau \overline{s}_* x = s_* x + s'_* x = 2 s_* x \in H_*(\conf_n)$
where the last equality follows because for $n$ odd we have that $s_*=s'_*$ (see Remark \ref{rem:s_sprime}).
%
\end{proof}

As a consequence of the previous result and Lemma \ref{lem:tau_invertible} we obtain the following.
\begin{cor} \label{cor:tau_restricted}
	If $p$ is an odd prime or $p=0$ and $\ring$ is a field of characteristic $p$ and $n$ is odd, then the projection on $H_i(\confmd{n},\conf_n)$ of the restriction of the homomorphism $\tau$ to $H_i(\confm{n},\conf_n)$
	$$
	\tau_|:H_i(\confm{n},\conf_n) \to H_i(\confmd{n},\conf_n)
	$$
	is an isomorphism.
\end{cor}
\begin{proof}
The corollary follows since under the stated conditions the homomorphism $\tau$ is an isomorphism (Lemma \ref{lem:tau_invertible}) and its restriction to $H_i(\conf_n)$ maps to $H_i(\conf_n)$.
\end{proof}

In order to prove the result concerning the odd torsion of the homology $H_i(\Br_n; \sym{g})$ we need to understand the maps 
$J$ in the diagram \eqref{diag:trerighe}.

We have the following result.
	
\begin{lem}\label{lem:ss_incl}
Let $E^2_{ij} =  H_j(\Br(n-i);\ring[t]/(1+t)) \Rightarrow H_i(\Gb{n}; \ring[t]/(1+t))$ be the spectral sequence induced by the
filtration $\Filt$ described above. The inclusion $\Ga{n-1} \into \Gb{n}$ induces the isomorphism of the term $H_j(\Ga{n-1}; \ring[t]/(1+t)) = H_j(\Br(n);\ring[t]/(1+t)) $ with the submodule $E^2_{0j}$ for all $j$.
\end{lem}
\begin{proof}
The Lemma follows since the inclusion $\conf_n \into \confm{n}$ maps the $i$-th standard generator of $\Ga{n-1}$ to the $(i+1)$-st standard generator of $\Gb{n}$. Hence the image of the complex that computes the homology of $\Ga{n-1}$ is the first term of the filtration $\Filt$.
\end{proof}
\begin{cor}\label{cor:image_mu}
According to the decomposition $H_*(\confm{n})= H_*(\conf_n) \oplus H_*(\confm{n}, \conf_n)$ induced by the section $\conf_n \into \confm{n}$  and the projection $\pi: \confm{n} \to \conf_n$ we have that the image of $\mu_*$ corresponds to $H_*(\confm{n}, \conf_n)$.
\end{cor}
\begin{proof}
This follows from Lemma \ref{lem:ss_incl} and Proposition \ref{prop:image_mu}, since, as seen in Remark \ref{rem:collapsss}, the spectral sequence of Lemma \ref{lem:ss_incl} collapses at the page $E^2$.
\end{proof}

\begin{thm}
Let $n$ be an odd integer and let $g = (n-1)/2$ and let $\sym{g}= H_1(\surf_n;\Z)$ be the integral symplectic representation of the braid group $\Br_n$.
Then the homology  $H_i(\Br_n; \sym{g})$ is a torsion $\Z$-module
with only $2^j$-torsion.
\end{thm}
\begin{proof}
From the description of the map $\mu_*$ (Proposition \ref{prop:isodellesemplicifazioni} and Corollary \ref{cor:image_mu}), the results about the map $\tau$ (Proposition \ref{prop:tau1}, Lemma \ref{lem:tau_invertible}, Corollary \ref{cor:tau_restricted}) and 
Remark \ref{rem:J} concerning the map $J$
we have that the map $\iota$ in diagram \eqref{diag:trerighe} is an isomorphism for $n$ odd and $p \neq 2$.
The result follows from the exact sequence of diagram \eqref{eq:bia2}.
\end{proof}

\section{A first bound for torsion order} \label{s:4tor}
In this section we focus on the torsion of order $2^j$ in the integral homology $H_{i}(\Br_n; H_1(\surf_n))$. 
We prove that torsion appears with order at most $4$. 

\begin{lem} \label{lem:no2tor1}
Let $\ring = \Z$. The homology $H_i(\confm{n}) \simeq H_i(\Gb{n}; \ring[t]/(1+t))$ has no $4$-torsion.
\end{lem}
\begin{proof}
As seen in Remark \ref{rem:collapsss}, we have a splitting
$$
H_q(\confm{n}; \Z) = \bigoplus_{i=0}^\infty H_{q-i} (\conf_{n-i}; \Z).
$$ 
Moreover we have that the integer cohomology of braid group has no $p^2$ torsion for any prime $p$ (\cite[Thm.~3]{vain}). Hence the result follows from the universal coefficients formula.
\end{proof}

\begin{lem} \label{lem:no2tor2}
Let $R = \Z$. For $n$ odd, the homology $H_i(\confmd{n}) \simeq H_i(\Gb{n}; \ring[t]/(1-t^2))$ has no $4$-torsion. 
\end{lem}
\begin{proof} 
It will suffice to show that the dimension over $\F_2$ of the homology of the complex $(H_*(\Gb{n}; \F_2[t]/(1-t^2)), \beta_2)$, where $\beta_2$ is  the Bockstein homomorphism, is the same as the dimension over $\Q$ of $H_*(\Gb{n}; \Q[t]/(1-t^2))$ (see \cite[Thm.~3E.4]{hatcher_at}).

According to \cite[Thm.~6.1, case $r=2$ and $n$ odd]{lehr04}, the Poincar\'e polynomial of the homology groups
$H_*(\Gb{n}; \Q[t]/(1-t^2)) 
$
is
$$
P(\Gb{n},t)= 
(1+t)(1+t+t^2+ \cdots + t^{n-1}). 
$$

The explicit computation of the Bockstein homomorphism $\beta_2$ of the homology group $H_*(\Gb{n}; \F_2[t]/(1-t^2))$ is the following. Let $$0 \to \Z_2[t]/(1-t^2) \stackrel{i_2}{\longrightarrow} \Z_4[t]/(1-t^2) \stackrel{\pi_2}{\longrightarrow} \Z_2[t]/(1-t^2) \to 0$$ be the short exact sequence of coefficients. Then (see Prop.~\ref{prop:generators})
$$(1-t) z_{c+1} x_{i_1} \cdots x_{i_k} = \pi_2 ((1-t) z_{c+1} x_{i_1} \cdots x_{i_k} )$$
and
$$
\DB (1-t) z_{c+1} x_{i_1} \cdots x_{i_k} = \sum_{
	\scsc{j=1, \ldots, k}{i_j>1}
}
2 (1-t) z_{c+1} x_{i_1} \cdots x_{i_j-1}^2 \cdots x_{i_k} =
$$
$$
= i_2 \left(\sum_{
	\scsc{j=1, \ldots, k}{i_j>1}
}
(1-t) z_{c+1} x_{i_1} \cdots x_{i_j-1}^2 \cdots x_{i_k} \right)
$$ and hence, using the notation introduced in Proposition \ref{prop:generators} we have
\begin{equation}\label{bockstein1}
\beta_2 \gp (z_c, x_0x_{i_1} \cdots x_{i_k})   =\sum_{
	\scsc{j=1, \ldots, k}{i_j>1}
} \gp (z_c, x_0 x_{i_1} \cdots x_{i_j-1}^2 \cdots x_{i_k}) 
\end{equation}
for all generators of the form given in \eqref{generators1}.
Moreover 
$$\frac{\DB(z_{c+1}x_{i_1} \cdots x_{i_k})}{(1+t)} = \pi_2 \left( \frac{1}{(1+t)}\left( \DB(z_{c+1}x_{i_1} \cdots x_{i_k}) - 2\sum_{
	\scsc{j=1, \ldots, k}{i_j>1}
}
z_{c+1} x_{i_1} \cdots x_{i_j-1}^2 \cdots x_{i_k} ) \right) \right) $$
and
$$
\DB \left( \frac{1}{(1+t)}\left( \DB(z_{c+1}x_{i_1} \cdots x_{i_k}) - 2\sum_{
		\scsc{j=1, \ldots, k}{i_j>1}
	}
	z_{c+1} x_{i_1} \cdots x_{i_j-1}^2 \cdots x_{i_k} ) \right) \right) =
$$
$$
=\DB \left( \frac{1}{(1+t)}\left(- 2\sum_{
	\scsc{j=1, \ldots, k}{i_j>1}
}
z_{c+1} x_{i_1} \cdots x_{i_j-1}^2 \cdots x_{i_k} ) \right) \right) =
$$
$$
= -2 \sum_{
	\scsc{j=1, \ldots, k}{i_j>1}
} \frac{\DB(z_{c+1} x_{i_1} \cdots x_{i_j-1}^2 \cdots x_{i_k} ) }{1+t}=
i_2 \left( \sum_{
	\scsc{j=1, \ldots, k}{i_j>1}
} \frac{\DB(z_{c+1} x_{i_1} \cdots x_{i_j-1}^2 \cdots x_{i_k} ) }{1+t} \right)
$$
hence we have
\begin{equation}\label{bockstein2}
\beta_2 \gs (z_c, x_0x_{i_1} \cdots x_{i_k}) =\sum_{
	\scsc{j=1, \ldots, k}{i_j>1}
}  \gs (z_c, x_{0} x_{i_1}  \cdots x_{i_j-1}^2 \cdots  x_{i_k})
\end{equation}
for all generators of the form given in \eqref{generators2}.

\begin{df} \label{def:modules}
Let $a,b$ be two non-negative integers, with  $a \in \N_{>0}$, $b \in \{0,1\}.$ 
Let $I = (i_1, \ldots, i_k)$, $J = (j_1, \ldots, j_h)$, where we assume that:
\begin{enumerate}[label=(\roman*)]
	\item $j_1 < \cdots <j_h$,
	\item $\min J \geq 2$,
	\item for all $s \in 1, \ldots, k$ there exists an integer $t \in 1, \ldots, h$ such that $i_s+1 = j_t$ 
\end{enumerate}  
We define the following sub-modules of $H_*(\Gb{n}; \F_2[t]/(1-t^2))$:
$$\MM{c,a,b,I,J} := \langle
\gs (z_c, x_0^a x_1^b x_{i_1}^2 \cdots x_{i_k}^2\epsilon(x_{j_1})\cdots \epsilon(x_{j_h})) | \mbox{ where }\epsilon(x_{j_t}) = x_{j_t} \mbox{ or }  x_{j_t-1}^2 \rangle.$$
and
$$\MMp{c,a,b,I,J} := \langle
\gp (z_c, x_0^a x_1^b x_{i_1}^2 \cdots x_{i_k}^2\epsilon(x_{j_1})\cdots \epsilon(x_{j_h})) | \mbox{ where }\epsilon(x_{j_t}) = x_{j_t} \mbox{ or }  x_{j_t-1}^2 \rangle.$$
\end{df}
\begin{rem} \label{rem:Jvuoto}
Notice that the condition (iii) above implies that if $J= \emptyset$ then also $I = \emptyset$ and hence $\MM{c,a,b,I,J}$ and $\MMp{c,a,b,I,J}$ have rank $1$ concentrated in degree $c+b$ and $c+b+1$ respectively. 
\end{rem}

The module $\MM{c,a,b,I,J}$ and $\MMp{c,a,b,I,J}$ are  free $\Z_2[t]/(1+t)$-modules, closed for $\beta_2$, as follows from formulas \eqref{bockstein1} and \eqref{bockstein2}.

If $J \neq \emptyset$, the complexes $(\MM{c,a,b,I,J}, \beta_2)$ and $(\MMp{c,a,b,I,J}, \beta_2)$ are acyclic. This fact can be proven by the same argument used in \cite[Lem.~4.4]{cal06}.
The argument can be expressed with the following statement:
\begin{lem} \label{lem:boolean}
Let $\bool{J}$ be the chain complex with $\F_2$ coefficients associated to the boolean lattice of the subsets of $J$. The complexes $(\MM{c,a,b,I,J}, \beta_2)$ and $(\MMp{c,a,b,I,J}, \beta_2)$ are isomorphic to $\bool{J}$. In particular if $J \neq \emptyset$ the complexes $(\MM{c,a,b,I,J}, \beta_2)$ and $(\MMp{c,a,b,I,J}, \beta_2)$ are acyclic.
\end{lem}
\begin{proof}[Proof of the Lemma \ref{lem:boolean}.]
Let first construct an isomorphism $\theta$ between the boolean complex $\bool{J}$ and the complex $\MM{c,a,b,I,J}$. We can map the generator $e_K$ of $\bool{J}$ associated to a subset $K$ of $J$ to the element $$ \theta(e_k):= \gs (z_c, x_0^a x_1^b x_{i_1}^2 \cdots x_{i_k}^2\epsilon(x_{j_1})\cdots \epsilon(x_{j_h})),$$ where $\epsilon(x_{j_t}) = x_{j_t}$ if $j_t \notin K$ and $\epsilon(x_{j_t}) =  x_{j_t-1}^2 $ if $j_t \in K$. It is easy to check that $\theta \circ d = \beta_2 \circ \theta$. Similarly we can construct an isomorphism $\theta': \bool{J} \to \MMp{c,a,b,I,J}$ with
$$
\theta'(e_J):= \gp (z_c, x_0^a x_1^b x_{i_1}^2 \cdots x_{i_k}^2\epsilon(x_{j_1})\cdots \epsilon(x_{j_h}))
$$
and check that $\theta' \circ d = \beta_2 \circ \theta'$.
\end{proof}

We recall from Proposition \ref{prop:generators} that the elements of the form $\gp (z_c, x_0 x_{i_1} \cdots x_{i_k})$ and $\gs (z_c, x_0 x_{i_1} \cdots x_{i_k})$ are a free set of generators of $H_*(\Gb{n}; \F_2[t]/(1-t^2))$.

We claim that for any pair of distinct modules $\MM{c,a,b,I,J}$ or $\MMp{c,a,b,I,J}$ we have disjoint set of generators.
The case of modules of the form $\MM{c,a,b,I,J}$ can be proved as follows (the case of modules of the form $\MMp{c,a,b,I,J}$ is analogous).  
Let
$$
\gs_0 = \gs (z_c, x_0^a x_1^b x_{i_1}^2 \cdots x_{i_k}^2\epsilon(x_{j_1})\cdots \epsilon(x_{j_h}))
$$
with $\epsilon(x_{j_t}) = x_{j_t} \mbox{ or }  x_{j_t-1}^2$ be a generator of $\MM{c,a,b,I,J}$. We can choose an element $\gs_1$ of the form 
$$
\gs_1 = \gs (z_c, x_0^a x_1^b x_{i_1}^2 \cdots x_{i_k}^2\epsilon'(x_{j_1})\cdots \epsilon'(x_{j_h}))
$$
such that $\gs_0$ appears as a summand in $\beta_2(\gs_1).$ The constrains on the multi-indexes $I$ and $J$ and formula \eqref{bockstein2} imply that such an element $\gs_1$ exists if and only for at least one index $l$ we have that $\epsilon(x_{j_l}) = x_{j_l-1}^2$. If such an element $\gs_1$ exists we have that $\gs_1 \in \MM{c,a,b,I,J}$ and we say that $\gs_0$ \emph{lifts} to $\gs_1$. Hence in a finite number of steps we have that $\gs_0$ lifts to $$\gs (z_c, x_0^a x_1^b x_{i_1}^2 \cdots x_{i_k}^2x_{j_1}\cdots x_{j_h}).
$$  This implies that an element $\gs_0$ does not belong at the same time to two different modules
$\MM{c,a,b,I,J}$ and $\MM{c',a',b',I',J'}.$ 

Therefore the modules $\MM{c,a,b,I,J}$ for all admissible $c,a,b,I,J$ are in direct sum and the modules $\MMp{c,a,b,I,J}$ for all admissible $c,a,b,I,J$ are in direct sum.
Moreover every generator 
$\gp (z_c, x_0 x_{l_1} \cdots x_{l_k})$ or $\gs (z_c, x_0 x_{l_1} \cdots x_{l_k})$ appears in a suitable 
complex  $\MM{c,a,b,I,J}$ or $\MMp{c,a,b,I,J}$. 
Let us prove this in the case of a generator of the form $\gp (z_c, x_0 x_{l_1} \cdots x_{l_k})$, the other case being analogous. We can write the monomial $x_0x_{l_1} \cdots x_{l_k}$ as
$$
x_{q_1}^{p_{q_1}} \cdots x_{q_r}^{p_{q_r}} 
$$
with $q_1 < q_2 < \cdots < q_r$. 
Then we define the strictly ordered multi-index $J$ as follows: $j \in J$ if and only if $j>1$ and one of the following conditions is satisfied:
\begin{enumerate}
	\item  $p_j$ is odd;
	\item $p_j=0$ and $p_{j-1}$ is even and non-zero.
\end{enumerate}
Moreover if $p_j$ is odd we set $\epsilon(x_j) = x_j$, otherwise we set $\epsilon(x_j) = x_{j-1}^2$.
Next we define the multi-index $I$ suitably in the unique way such that 
$$
x_{q_1}^{p_{q_1}} \cdots x_{q_r}^{p_{q_r}} = x_0^{p_0} x_1^{p_1} x_{i_1}^2 \cdots x_{i_k}^2\epsilon(x_{j_1})\cdots \epsilon(x_{j_h}).
$$
It is straightforward to check that 
$$
\gs (z_c, x_0^{p_0} x_1^{p_1} x_{i_1}^2 \cdots x_{i_k}^2\epsilon(x_{j_1})\cdots \epsilon(x_{j_h})) = \gs (z_c,  x_{q_1}^{p_{q_1}} \cdots x_{q_r}^{p_{q_r}} ) = 
$$
$$
=\gs (z_c, x_0 x_{l_1} \cdots x_{l_k})
$$
and that the multi-indexes $I$ and $J$ satisfies the condition of Definition \ref{def:modules}.

Hence 
$h_{i}(n,2)$ is the direct sum of all admissible modules $\MM{c,a,b,I,J}$ and 
$h_{i}'(n,2)$ is
the direct sum of all admissible modules $\MMp{c,a,b,I,J}$ and we recall (equation \eqref{eq:decomposizione}) that
$$
H_*(\Gb{n}; \F_2[t]/(1-t^2)) = h_{i}(n,2) \oplus h_{i}'(n,2).
$$
This direct sum decomposition  implies that the homology $H_{\beta_2}$ of the complex $(H_*(\Gb{n};$ $\F_2[t]/(1-t^2)), \beta_2)$ is given as follows:
$$
H_{\beta_2} = \bigoplus \MM{c,a,b,\emptyset, \emptyset}
\oplus
\bigoplus \MMp{c,a,b,\emptyset, \emptyset}
$$
for $c$ even, $a \in \N_{>0}$, $b \in \{0,1\}$. In fact if $J = \emptyset$ then also $I = \emptyset$ and for all non-empty $J$ we have from Lemma \ref{lem:boolean} that the complexes $(\MM{c,a,b,I,J},\beta_2)$ and $(\MMp{c,a,b,I,J},\beta_2)$ are acyclic.

Using Remark \ref{rem:Jvuoto}  it is easy to check that the complex $H_{\beta_2}$ has Poincar\'e polynomial $(1+t)(1+t+t^2+ \cdots + t^{n-1})$, hence the lemma follows.
\end{proof}

\begin{lem} \label{lem:tau_torsion}
Let $\ring = \Z$. For $n$ odd the cokernel of the homomorphism
$$\tau:H_i(\confm{n}) \to H_i(\confmd{n})
$$
has no $4$-torsion.
\end{lem}
\begin{proof}
We recall that in Definition \ref{def:BB} we introduced the sets $\cB' \subset H_i(\confm{n};\Z)$  and $\cB'' \subset H_i(\confmd{n};\Z)$. As stated in Proposition \ref{prop:BB}, these are free sets of generators of a maximal free $\Z$-submodule of $H_i(\confm{n};\Z)$  and $H_i(\confmd{n};\Z)$ respectively.

To describe the cokernel we see that the map 
$$\tau:H_i(\confm{n}; \Z) \to H_i(\confmd{n}; \Z)
$$ 
acts on the elements of $\cB'$ as follows:
\begin{eqnarray}
\tau: \omega_{2i,j,0}^{(1)} & \mapsto  (1 - t ) &\omega_{2i,j,0}^{(2)} = 2 \omega_{2i,j,0}^{(2)}, \label{rat_gen1}\\
\tau: \widetilde{\omega}_{2i,j,0}^{(1)} & \mapsto \phantom{(1-t)}& \widetilde{\omega}_{2i,j,0}^{(2)},\label{rat_gen2}\\
\tau: \omega_{2i,j,1}^{(1)}  & \mapsto (1 - t )& \omega_{2i,j,1}^{(2)} = 2 \omega_{2i,j,1}^{(2)},\label{rat_gen3}\\
\tau: \widetilde{\omega}_{2i,j,1}^{(1)} & \mapsto \phantom{(1-t)}& \widetilde{\omega}_{2i,j,1}^{(2)}.\label{rat_gen4}
\end{eqnarray}
Hence the homomorphism $\tau$ acts diagonally and maps each element of $\cB'$ to $1$ or $2$ times the corresponding element of $\cB''$.

Since the $\Z$-modules $H_i(\confm{n};\Z)$ and $H_i(\confmd{n};\Z)$ have no $4$-torsion and  $\tau$ is an isomorphism mod $p$ for any odd prime, it follows that, with integer coefficients, the cokernel of $\tau$ has no $4$-torsion.
\end{proof}

Let us consider homology with coefficient in the ring $\ring = \Z$. As stated in Corollary \ref{cor:image_mu}, the image of  $\mu_*$ is the submodule $H_i(\confm{n}, \conf_n)$ in $H_i(\confm{n})$ and we consider the composition $\iota = J \circ \tau \circ \mu_*: H_{i-1}(\confm{n-1}) \otimes H_1(S^1) \to H_i(\confmd{n}, \conf_n)$.

\begin{lem} \label{lem:comp_no_4}
For odd $n$ the cokernel of the composition $J \circ \tau \circ \mu_*$ has no $4$-torsion.
\end{lem}
\begin{proof}
First we can consider the homomorphism $\overline{s}_*:H_*(\conf_n) \to H_*(\confm{n})$ induced by the inclusion $\overline{s}:\conf_n \into \confm{n}$, the decomposition $H_*(\confmd{n}) = H_*(\conf_n) \oplus H_*(\confmd{n}, \conf_n)$, with $\pi_1$ and $\pi_2= J$ respectively the projections on the first and the second summand,  and hence the map  $\tau_{11}:H_*(\conf_n) \to H_*(\conf_n)$ defined by the composition
$$
H_*(\conf_n) \stackrel{\overline{s}_*}{\longrightarrow} H_*(\confm{n}) \stackrel{\tau}{\longrightarrow} H_*(\confmd{n}) \stackrel{\pi_1}{\longrightarrow} H_*(\conf_n).
$$
We can consider the following diagram:
$$
\xymatrix{ H_*(\conf_n) \ar[d]^{\overline{s}_*} \ar[drr]^(0.65){s_*+s'_*} \ar[rr]_(0.44){\tau_{|\overline{s}_*H_*(\conf_n)}} \ar@/^1pc/[rrr]^{\tau_{11}}
	&& H_*(s(\conf_{n})  \sqcup s'(\conf_{n}))\ar[d]^{i_*} \ar[r] & H_*(\conf_n) \\ H_*(\confm{n}) \ar[rr]^\tau && H_*(\confmd{n}) \ar[ur]^{\pi_1}  }
$$
From Theorem \ref{teo:tau_restricted} we have that that $\tau_{11} = 2 \Id_{H_*(\conf_n)}$ and
$\pi_2 \tau\overline{s}_*(H_*(\conf_n))=0$. 
Now, let $x_2 \in H_*(\confmd{n}, \conf_n)$ and let $\tau_{22}: H_*(\confm{n}, \conf_n) \to H_*(\confmd{n}, \conf_n)$ be the map induced by $\tau$ by restricting to $H_*(\confm{n}, \conf_n)$ and projecting to $H_*(\confmd{n}, \conf_n)$. If there exists $y \in H_*(\confm{n})$ such that $\pi_2 \tau (y) = 4 x_2$, then let $x_1:= \pi_1(\tau(y))$. We can consider $-x_1 + 2 y \in H_*(\confm{n})$ and we have that $\tau(-x_1 + 2 y) = -2 x_1 + 2 (x_1 +4x_2) = 8 x_2$. Since the cokernel of $\tau$ has only $2$-torsion it follows that $2x_2 = 0$ in $\coker \tau$ and finally, since $\pi_2 \tau \overline{s}_*(H_*(\conf_n))= 0$, $2x_2 = 0$ in $\coker \tau_{22}$.
\end{proof}


From Lemma \ref{lem:no2tor1} and \ref{lem:comp_no_4} we have that, with integer coefficients, the kernel and the cokernel of the map $$\iota:H_{i-1}(\Br_n; H_1(S^1 \times P)) \to H_{i-1}(\Br_n; H_1(\ddiskP))$$ in diagram \eqref{eq:bia2} have no $4$-torsion. 
Hence we have (see also \cite{bianchi}):
\begin{thm}\label{thm:4tor}
For $n$ odd the homology $H_{i}(\Br_n; H_1(\surf_n))$ computed with coefficients in the ring $\ring = \Z$ has torsion of order at most $4$. \qed
\end{thm}

\section{No $4$-torsion}\label{sec:no4tor}
In this section we will show that the homology $H_{i}(\Br_n; H_1(\surf_n))$ for $n$ odd actually has only $2$-torsion.

In order to prove this we will consider the following short exact sequence associated to \eqref{eq:bia2}, with coefficients in $\Z_2$ and in $\Z$.
$$
0 \to \coker \iota \to H_{i}(\Br_n; H_1(\surf_n)) \to \ker \iota \to 0
$$

Let us fix the number $n$. From Theorem \ref{thm:4tor} we can assume that $H_{i}(\Br_n; H_1(\surf_n; \Z)) = \Z_2^{a_i} \oplus \Z_4^{b_i}$.
Since the modules $\ker \iota$ and $\coker \iota$ have no $4$-torsion, we can assume $$\coker (\iota_i :H_{i}(\Br_n; H_1(S^1 \times P;\Z)) \to H_{i}(\Br_n; H_1(\ddiskP;\Z) ) = \Z_2^{u_i}$$ and $$\ker (\iota_i :H_{i}(\Br_n; H_1(S^1 \times P;\Z)) \to H_{i}(\Br_n; H_1(\ddiskP;\Z) ) = \Z_2^{v_i}$$
and clearly we have 
$$
u_i + v_{i-1} = a_i + 2 b_i.
$$
Moreover, with coefficients in $\Z_2$, we have
$$
H_{i}(\Br_n; H_1(\surf_n; \Z_2)) = \Z_2^{a_i+a_{i-1} + b_i + b_{i-1}}.
$$
Let $$\overline{u}_i := \rk \coker (\iota_i :H_{i}(\Br_n; H_1(S^1 \times P ;\Z_2)) \to H_{i}(\Br_n; H_1(\ddiskP ;\Z_2) ) $$
and 
$$\overline{v}_i := \rk \ker (\iota_i :H_{i}(\Br_n; H_1(S^1 \times P ;\Z_2)) \to H_{i}(\Br_n; H_1(\ddiskP;\Z_2) ). $$
It follows that $H_*(\Br_n; H_1(\surf_n; \Z))$ has no $4$-torsion if and only if 
$$2 \sum_i (u_i + v_i) = \sum_i (\overline{u}_i  + \overline{v}_i ). $$

Hence we can compute the rank of the modules above.

A basis of the homology $H_{i}(\Br_n; H_1(S^1 \times P;\Z_2))$ is given as follows. 
Following \cite{calmar}, the homology $H_*(\Gb{n}; \F_2[t]/(1+t))$ for $n$ odd is generated, as an $\F_2[t]$-module, by the classes of the form
\begin{equation}\label{generators1Z2}
\gpu (z_c, x_0 x_{i_1} \cdots x_{i_k}) :=z_{c+1} x_{i_1} \cdots x_{i_k}
\end{equation}
 and
\begin{equation}\label{generators2Z2}
\gsu(z_c, x_0 x_{i_1} \cdots x_{i_k}):=\frac{\DB(z_{c+1}x_{i_1} \cdots x_{i_k})}{(1+t)}
\end{equation}
where we assume $0 \leq i_1 \leq \cdots i_k$ and $c$ even.

In particular the image of $\mu_*$ if generated by all elements $\gpu (z_c, x_0 x_{i_1} \cdots x_{i_k})$ and all elements $\gsu(z_c, x_0 x_{i_1} \cdots x_{i_k})$ with $c>0$.

As seen in Section \ref{s:4tor} the homology $H_*(\Gb{n}; \F_2[t]/(1-t^2))$ for $n$ odd is generated, as an $\F_2[t]$-module, by the following classes already introduced in \eqref{generators1} and \eqref{generators2}:
\begin{equation*}
\gp (z_c, x_0 x_{i_1} \cdots x_{i_k}) :=(1-t)z_{c+1} x_{i_1} \cdots x_{i_k}
\end{equation*}
and
\begin{equation*}
\gs(z_c, x_0 x_{i_1} \cdots x_{i_k}):=\frac{\DB(z_{c+1}x_{i_1} \cdots x_{i_k})}{(1+t)}
\end{equation*}
where we assume $0 \leq i_1 \leq \cdots i_k$ and $c$ even.
The kernel of $J$ is generated by the classes 
$\gs(z_c, x_0 x_{i_1} \cdots x_{i_k})$ for $c=0$ as one can see that these classes generate the image of $s_*$.

The map $\tau$ acts as follows:
\begin{eqnarray*}
\tau: \gpu (z_c, x_0 x_{i_1} \cdots x_{i_k}) & \mapsto & \gp (z_c, x_0 x_{i_1} \cdots x_{i_k});\\
\tau: \gsu (z_c, x_0 x_{i_1} \cdots x_{i_k}) &\mapsto & 0.
\end{eqnarray*}

\begin{rem}\label{rem:basi_mod_2}
A basis of $\coker (\iota_i :H_{i}(\Br_n; H_1(S^1 \times P;\Z_2)) \to H_{i}(\Br_n; H_1(\ddiskP;\Z_2) )$ is
given by the elements $\gs (z_c, x_0 x_{i_1} \cdots x_{i_k})$ with $c>0$, of degree $n$ and homological dimension $i+1$, while a basis of $ \ker (\iota_i :H_{i}(\Br_n; H_1(S^1 \times P;\Z_2)) \to H_{i}(\Br_n; H_1(\ddiskP;\Z_2) )$ is given by the elements $\gsu (z_c, x_0 x_{i_1} \cdots x_{i_k})$ with $c>0$, of degree $n$ and homological dimension $i+1$.
In particular, in order to have odd degree $c$ must be even.
Clearly this two basis are in bijection and we have $\overline{u}_i = \overline{v}_i$.
\end{rem}
 

In order to describe the corresponding map with integer coefficient we recall that in Section \ref{s:4tor} we described basis
 $\cB'$, $\cB''$ generating the homology of $H_i(\confm{n};\Q) $ and $H_i(\confmd{n};\Q)$ and spanning a maximal free $\Z$-submodule of $H_i(\confm{n};\Z) $ and $H_i(\confmd{n};\Z)$. The action of $\tau$ with respect to these basis is given in equations (\ref{rat_gen1}--\ref{rat_gen4}). The elements of $\cB'$ (resp.~$\cB''$) of the form $\omega_{2i,j,\epsilon}^{(1)}$ (resp.~$\omega_{2i,j,\epsilon}^{(2)}$) map, modulo $2$, to elements of the form $\gsu$ (resp.~$\gs$) and in particular the elements the form $\omega_{2i,j,\epsilon}^{(1)}$ (resp.~$\omega_{2i,j,\epsilon}^{(2)}$) with $i=0$ map to elements of the form $\gsu(z_c, \ldots)$ (resp.~$\gs(z_c, \ldots)$) with $c=0$.
The elements of $\cB'$ (resp.~$\cB''$) of the form $\widetilde{\omega}_{2i,j,\epsilon}^{(1)}$ (resp.~$\widetilde{\omega}_{2i,j,\epsilon}^{(2)}$) map, modulo $2$, to elements of the form $\gpu$ (resp.~$\gp$).
Let $$w_i = | \{ \omega_{0,j,\epsilon}^{(1)} \in H_i(\confm{n};\Q)\} | = 
| \{ \omega_{0,j,\epsilon}^{(2)} \in H_i(\confmd{n};\Q)\} |.
$$
From the Universal Coefficients Theorem and from the description of $\tau$ given in equations (\ref{rat_gen1}--\ref{rat_gen4}) we have that
$$
\sum_i u_i = \sum_i  \frac{ \overline{u}_i - w_i}{2} +\sum_i  w_i
$$
and
$$
\sum_i v_i =  \sum_i  \frac{\overline{v}_i - w_i}{2}.
$$
Then it is straightforward to see that 
$$2 \sum_i (u_i + v_i) = \sum_i (\overline{u}_i  + \overline{v}_i ) $$
and hence we have proved the following result.
\begin{thm} \label{th:no4tor}
For odd $n$ the homology $H_{i}(\Br_n; H_1(\surf_n;\Z))$ has only $2$-torsion. \qed
\end{thm}



\section{Stabilization and computations}\label{sec:stab}

\subsection{Stablization results}
Applying the results of Wahl and Randal-Willians (\cite{W-RW}) for the stability of family of groups with twisted coefficients  it is possible to prove  that the groups $H_i(\Br_n; H_1(\surf_n))$ stabilize for all $i$. 
In particular (see see \cite[Thm.~52]{bianchi}) the map $H_i(\Br_n; H_1(\surf_n)) \to H_i(\Br_{n+1}; H_1(\surf_{n+1}))$ is an epimorphism for $i \leq \frac{n}{2}-1$ and an isomorphism for $i \leq \frac{n}{2}-2$.
In this section  we will prove a slightly sharper result using the explicit description of the homology.

\begin{df}
	The stabilization map $\stab: \confm{n} \to \confm{n+1}$ is given by 
	$$
	(p_1, \{p_2, \ldots, p_{n+1}\}) \mapsto (p_1, \{p_2, \ldots, p_{n+1}, \frac{1+\max_{1 \leq i \leq n+1} (|p_i|)}{2} \}) 
	$$
\end{df}

We recall that $\confm{n}$ is a classifying space for $\Gb{n}$, that is the Artin group of type $\mathrm{B}$. Moreover we recall from \cite[Cor.~4.17, 4.18,4.19]{calmar} (notations $H_i(\mathrm{B}(2,1,n); \F))$ and $H_i(\mathrm{B}(4,2,n); \F))$ were there used respectively for $H_i(\Gb{n}; \F[t]/(1+t))$ and $H_i(\Gb{n}; \F[t]/(1-t^2))$):
\begin{prop}\label{prop:stab_rank}
Let $p$ be a prime or $0$. Let $\F$ be a field of characteristic $p$. Let us consider the stabilization homomorphisms
$$
\stab_*:H_i(\Gb{n}; \F[t]/(1+t)) \to H_i(\Gb{n+1}; \F[t]/(1+t))
$$
and
$$
\stab_*:H_i(\Gb{n}; \F[t]/(1-t^2)) \to H_i(\Gb{n+1}; \F[t]/(1-t^2)).
$$
\begin{enumerate}[label=\alph*)]
\item If $p=2$ $\stab_*$ is epimorphism for $2i \leq n$ and isomorphisms for $2i < n$.
\item If $p>2$ $\stab_*$ is epimorphism for $\frac{p(i-1)}{p-1}+2 \leq n$ and isomorphisms for $\frac{p(i-1)}{p-1}+2 < n$.
\item If $p=0$ $\stab_*$ is epimorphism for $i+1 \leq n$ and isomorphisms for $i+1 < n$.
\end{enumerate}
\end{prop}

The map $\mu$ commutes, up to homotopy, with the stabilization map $\stab: \confm{n} \to \confm{n+1}$:
$$
\xymatrix{
\confm{n-1} \times S^1 \ar[d]^{\stab \times \Id} \ar[r]^\mu& \confm{n} \ar[d]^\stab\\
\confm{n} \times S^1 \ar[r]^\mu & \confm{n+1}
}
$$


The map $\tau$ naturally commutes with the stabilization homomorphism $\stab_*$ in homology, since $\tau$ is given by the multiplication by $(1-t)$.


We can also define a geometric stabilization map $\stab: \confmd{n} \to \confmd{n+1}$ as follows:
$$
\gstab:(P, z, y) \mapsto (P \cup \{p_{\infty}\}, z, y \sqrt{z-p_{\infty}})
$$
where we set $p_\infty:= \frac{\max ( \{|p_i|, p_i \in P \} \cup \{|z|\})+1}{2}$ and since $\Re(z-p_{\infty}) < 0$ we choose
$\sqrt{z-p_{\infty}}$ to be the unique square root with $\Im(\sqrt{z-p_{\infty}}) >0$

The following diagram is homotopy commutative:
$$
\xymatrix{
	\conf_n \ar[d]^{\stab} \ar[r]^s& \confmd{n} \ar[d]^\gstab\\
	\conf_{n+1}  \ar[r]^s & \confmd{n+1}
}
$$
and this imply that $J$ commutes with the stabilization homomorphism $\gstab_*$.

We also need to prove that the following diagram commutes:
$$
\xymatrix{
	H_*(\Gb{n};\ring[t]/(1-t^2)) \ar[d]^{\stab_*} \ar[r]^(0.63){\simeq}& H_*(\confmd{n}) \ar[d]^{\gstab_*}\\
	H_*(\Gb{n+1};\ring[t]/(1-t^2))  \ar[r]^(0.63){\simeq} & H_*(\confmd{n+1})
}
$$
This is true since the homomorphism $$\stab_*:H_*(\Gb{n};\ring[t]/(1-t^2)) \to H_*(\Gb{n+1};\ring[t]/(1-t^2))$$ is induced by the map $\stab: \confm{n} \to \confm{n+1}$ previously defined and it is obtained applying the Shapiro lemma to $\confm{n} = k(\Gb{n},1)$, with $\ring[t]/(1-t^2)) = \ring[\Z_2] = \ring[\pi_1(\confm{n})/\pi_1(\confmd{n})]$.
It is straightforward to check that the diagram
$$
\xymatrix{
 \confmd{n} \ar[d]^\gstab \ar[r]& 	\confm{n}\ar[d]^{\stab} \\
 \confmd{n+1} \ar[r] & 	\confm{n+1}
}
$$
commutes, where the horizontal maps are the usual double coverings. Finally we have proved the following result.
\begin{lem}\label{lem:commutative_stab}
The following diagram is commutative
\begin{equation}\label{diag:stabilization}
\begin{tabular}{c}
\xymatrix @R=2pc @C=2pc {
H_{i-1}(\confm{n-1}) \otimes H_1(S^1) \ar[d]^{\stab_* \otimes \Id} \ar[r]^(0.6){\mu_*}  &	H_i(\confm{n}) \ar[r]^{\tau} \ar[d]^{\stab_*} &H_i(\confmd{n})  \ar[r]^(0.45)J \ar[d]^{\gstab_*}& H_i(\confmd{n}, \conf_n) \ar[d]^{\gstab_*} \\
H_{i-1}(\confm{n}) \otimes H_1(S^1) \ar[r]^(0.6){\mu_*}  &	H_i(\confm{n+1}) \ar[r]^{\tau} &H_i(\confmd{n+1})  \ar[r]^(0.45)J & H_i(\confmd{n+1}, \conf_{n+1}) 
}
\end{tabular}
\end{equation}
\end{lem}

\begin{thm}\label{thm:stabilization}
Consider homology with integer coefficients. The homomorphism 
$$
H_i(\Br_n; H_1(\surf_n)) \to H_i(\Br_{n+1}; H_1(\surf_{n+1}))
$$
is an epimorphism for $i \leq \frac{n}{2}-1 $
and an isomorphism for $i < \frac{n}{2}-1$.

For $n$ even $H_i(\Br_n; H_1(\surf_n))$ has no $p$ torsion (for $p > 2$) when $\frac{pi}{p-1}+3 \leq n$ and no free part for $i+3 \leq n$. In particular for $n$ even,  when $\frac{3i}{2}+3 \leq n$ the group $H_i(\Br_n; H_1(\surf_n))$ has only $2$-torsion.
\end{thm}
\begin{proof}
The maps in the diagram  \eqref{diag:stabilization} with $\Z_p$ coefficients fits in the map of long exact sequences
$$
\xymatrix @R=1.5pc @C=0.8pc {
\cdots \ar[r] &	H_{i-1}(\confm{n-1};\Z_p) \otimes H_1(S^1) \ar[d]^{\stab_* \otimes \Id} \ar[r]^(0.6){\iota}  &	H_i(\confmd{n}, \conf_n;\Z_p) \ar[d]^{\gstab_*} \ar[r] & H_{i-1} (\Br_n; H_1(\surf_n;\Z_p)) \ar[d]^{\stab_*} \ar[r] & \cdots\\
\cdots \ar[r] &	H_{i-1}(\confm{n};\Z_p) \otimes H_1(S^1) \ar[r]^(0.55){\iota}  & H_i(\confmd{n+1}, \conf_{n+1};\Z_p) \ar[r]  & H_{i-1} (\Br_{n+1}; H_1(\surf_n;\Z_p)) \ar[r] & \cdots
}
$$
that near $H_{i-1} (\Br_n; H_1(\surf_n;\Z_p))$
looks as follows:
\begin{equation}\label{diag:stabilization2}
\begin{tabular}{l}
\xymatrix @R=1.5pc @C=0.8pc {
	\!\!\cdots\ar[r]\! &	\!H_i(\confmd{n}, \conf_n;\Z_p)\! \ar[d]^{\gstab_*} \ar[r] &\!H_{i-1} (\Br_n; H_1(\surf_n;\Z_p))\! \ar[d]^{\stab_*} \ar[r] &\!H_{i-2}(\confm{n-1};\Z_p)\! \otimes\!H_1(S^1)\! \ar[d]^{\stab_* \otimes \Id} \ar[r] &\!\cdots\!\!\\
	\!\!\cdots\!\ar[r]\!& \!H_i(\confmd{n+1}, \conf_{n+1};\Z_p)\! \ar[r]  & \!H_{i-1} (\Br_{n+1}; \!H_1(\surf_n;\Z_p))\!\ar[r] &	\!H_{i-2}(\confm{n};\Z_p) \! \otimes\!H_1(S^1)\!\ar[r]  & \!\cdots\!\!
}
\end{tabular}
\end{equation}
For $p=2$, from Proposition \ref{prop:stab_rank} and Lemma \ref{lem:commutative_stab} we have that the vertical map $\gstab_*$ on the left of diagram \eqref{diag:stabilization2} is an epimorphism for for $i \leq \frac{n}{2}$ and isomorphisms for $i < \frac{n}{2}$. The vertical map $\stab_* \otimes \Id$ on the right of diagram \eqref{diag:stabilization2} is an isomorphism for $i \leq \frac{n}{2}$.

This implies that $$
\stab_*:H_i(\Br_n; H_1(\surf_n; \Z_2)) \to H_i(\Br_{n+1}; H_1(\surf_{n+1}; \Z_2))
$$
is epimorphism for $i \leq \frac{n}{2} -1$
and an isomorphism for $i < \frac{n}{2}-1$. 

For $p>2$, from Proposition \ref{prop:stab_rank} and Lemma \ref{lem:commutative_stab} we have that the vertical map $\gstab_*$ on the left of diagram \eqref{diag:stabilization2} is an epimorphism for $\frac{p(i-1)}{p-1}+2 \leq n$ and isomorphisms for $\frac{p(i-1)}{p-1}+2 < n$. The vertical map $\stab_* \otimes \Id$ on the right of diagram \eqref{diag:stabilization2} is an isomorphism  for $\frac{p(i-1)}{p-1}+2 \leq n$. We notice that actually these conditions are weaker that the ones for $p=2$.

This implies that for $p>2$ $$
\stab_*:H_i(\Br_n; H_1(\surf_n; \Z_p)) \to H_i(\Br_{n+1}; H_1(\surf_{n+1}; \Z_p))
$$
is epimorphism for $\frac{pi}{p-1}+2 \leq n$
and an isomorphism for $\frac{pi}{p-1}+2 < n$. 

The same argument for $p=0$ shows that
$$
\stab_*:H_i(\Br_n; H_1(\surf_n; \Q)) \to H_i(\Br_{n+1}; H_1(\surf_{n+1}; \Q))
$$
is an epimorphism for $i+2 \leq n$
and an isomorphism for $i+2< n$. 

From the Universal Coefficients Theorem for homology we get that the homomorphism 
$$
H_i(\Br_n; H_1(\surf_n)) \to H_i(\Br_{n+1}; H_1(\surf_{n+1}))
$$
is an epimorphism for $i \leq \frac{n}{2}-1 $
and an isomorphism for $i < \frac{n}{2}-1$.

Since the integer  homology $H_i(\Br_n; H_1(\surf_n))$ has only $2$-torsion for $n$ odd, 
the stabilization implies that for $n$ even $H_i(\Br_n; H_1(\surf_n))$ has no $p$ torsion (for $p > 2$) for $\frac{pi}{p-1}+3 \leq n$ and no free part for $i+3 \leq n$.

In particular, for $n$ odd,  $\frac{3i}{2}+3 \leq n$ we have that $H_i(\Br_n; H_1(\surf_n))$ has only $2$-torsion.
\end{proof}

\begin{thm}\label{thm:unstable}
For $n$ even the groups $H_i(\Br_n;H_1(\surf_n;\Z))$ are torsion, except for $i=n-1, n-2$ where $H_i(\Br_n;H_1(\surf_n;\Q)) = \Q.$
\end{thm}
\begin{proof}
The result follows from Proposition \ref{prop:coker_n_pari} and, for $i<2$, from the stabilization Theorem \ref{thm:stabilization}. For $n=4$ the result follows from a direct computation.
\end{proof}

In Table \ref{tab:conti} we present some computations of the groups $H_i(\Br_n;H_1(\surf_n))$ (with integer coefficients). The computations are obtained using an Axiom implementation of the complex introduced in \cite{salvetti}.
\begin{table*}[htb] 
\renewcommand{\arraystretch}{1.2}
\setlength\extrarowheight{3pt}
\begin{center}
\begin{tabular}{|c|c|c|c|c|c|c|c|c|c|c|c|}
\hline
\backslashbox{$n$}{$i$}
& 1 & 2 & 3 & 4 & 5 & 6 & 7 & 8 & 9 & 10 & 11\\
\hline
$3$ & $\Z_2$ &&&&&&&&&&\\
\hline
$4$ & $\Z_2^2$ & $\Z$&$\Z$&&&&&&&&\\
\hline
$5$
 & 
\cellcolor{lightgray!30}
$\Z_2$ & 
$\Z_2$
&$\Z_2$&&&&&&&&\\ 
\hline
$6$ & $\Z_2$ & $\Z_2^2$&$\Z_2^2  \Z_3$&$\Z$&$\Z$&&&&&&\\
\hline
$7$ & $\Z_2$ & \chl $\Z_2$&$\Z_2^2 $&$\Z_2^2$&$\Z_2$&&&&&&\\
\hline
$8$ & $\Z_2$ & $\Z_2$&$\Z_2^3 $&$\Z_2^3\Z_3$&$\Z_2^3\Z_3$&$\Z$&$\Z$&&&&\\
\hline
$9$ & $\Z_2$ & $\Z_2$&\chl $\Z_2^2 $&$\Z_2^3$&$\Z_2^3$&$\Z_2^2$&$\Z_2$&&&&\\
\hline
${10}$ & $\Z_2$ & $\Z_2$&$\Z_2^2 $&$\Z_2^4$&$\Z_2^4$&$\Z_2^4\Z_3$&$\Z_2^3 \Z_3\Z_5$&$\Z$&$\Z$&&\\
\hline
${11}$ & $\Z_2$ & $\Z_2$&$\Z_2^2 $&\chl $\Z_2^3$&$\Z_2^4$&$\Z_2^4$&$\Z_2^4 $&$\Z_2^3$&$\Z_2$&&\\
\hline
${12}$ & $\Z_2$ & $\Z_2$&$\Z_2^2 $&$\Z_2^3$&
$\Z_2^5$ & $\Z_2^5$
&$\Z_2^6\Z_3$&$\Z_2^6\Z_3\Z_5$&$\Z_2^3\Z_3\Z_5$&$\Z$&$\Z$\\
\hline
${13}$ & $\Z_2$ & $\Z_2$&$\Z_2^2 $&$\Z_2^3$&\chl $\Z_2^4$&$\Z_2^5$&$\Z_2^6$&$\Z_2^6$&$\Z_2^5$&$\Z_2^3$&$\Z_2$\\
\hline
\end{tabular}
\end{center}
\caption{Computations of $H_i(\Br_n;H_1(\surf_n))$. For each column the first stable group is highlighted.}\label{tab:conti}
\end{table*}

\subsection{Poincar\'e polynomials}
We use the previous results to compute explicitly the Poincar\'e polynomials for odd $n$.


\begin{thm}\label{thm:poincare}
For odd $n$, 
the rank of $H_i(\Br_n;H_1(\surf_n;\Z))$ as a $\Z_2$-module is the coefficient of $q^it^n$ in the series
 $$
\widetilde{P}_2(q,t)=\frac{qt^3}{(1-t^2q^2)} \prod_{i \geq 0} \frac{1}{1-q^{2^i-1}t^{2^i}}
$$
In particular the series $\widetilde{P}_2(q,t)$ is 
the Poincar\'e series of the homology group  
$$\bigoplus_{n \mbox{\scriptsize odd}}H_*(\Br_{n};H_1(\surf_n;\Z))$$
as a $\Z_2$-module. 
\end{thm}

\begin{proof}
Let $P_2(\Br_n,H_1(\surf_n))(q)$ be the Poincar\'e polynomial for $H_*(\Br_n;H_1(\surf_n;\Z))$ as a $\Z_2$-module. 
Since we already know that for $n$ odd the homology $H_i(\Br_n;H_1(\surf_n;\Z))$ has only $2$-torsion, we can obtain $P_2(\Br_n,H_1(\surf_n))(q)$ from the Universal Coefficients Theorem as follows. We compute the Poincar\'e polynomial 
$
P_2(\Br_n, H_1(\surf_n;\Z_2))(q)
$
for
$H_i(\Br_n;H_1(\surf_n;\Z_2))
$
and we divide by $1+q$.

In order to compute $
P_2(\Br_n, H_i(\surf_n;\Z_2))(q)
$ 
we consider the short exact sequence
$$
0 \to \coker \iota_i \to H_{i}(\Br_n; H_1(\surf_n:\Z_2)) \to \ker \iota_{i-1} \to 0
$$
where we recall that $\iota_i$ is the map
$$\iota_i:H_{i}(\Br_n; H_1(S^1 \times P:\Z_2)) \to H_{i}(\Br_n; H_1(\ddiskP:\Z_2)).$$
Now it follows from the Remark \ref{rem:basi_mod_2} and from the description in Section \ref{sec:no4tor} that for a fixed odd $n$ the ranks of $\ker \iota$ and $\coker \iota$
are respectively the coefficients of $q^it^n$ in the following series:
$$
P_2(\coker \iota) = P_2(\ker \iota)  =
\frac{qt^3}{1-t^2q^2} \prod_{j \geq 0} \frac{1}{1-q^{2^j-1}t^{2^j}}
$$
Clearly the Polynomial for 
$H_{i}(\Br_n; H_1(\surf_n:\Z_2))$ is given by the coefficient of $t^n$ in the sum
$$
 P_2(\coker \iota) + qP_2(\ker \iota) 
$$
and hence dividing by $(1+q)$ we get our result. 
\end{proof}

The same argument of the previous proof can be applied in stable rank. From the Remark \ref{rem:basi_mod_2} the Stable Poincar\'e polynomial of both $\coker \iota$ and $\ker \iota$ with $\Z_2$ coefficients is the following:
$$
\frac{q}{1-q^2} \prod_{j \geq 1} \frac{1}{1-q^{2^j-1}}
$$
Since there is no free part in $H_i(\Br_n;H_1(\surf_n;\Z))$ all these groups have only $2$-torsion. In particular for integer coefficients we get the following statement.
\begin{thm}\label{thm:stablepoincare}
The Poincar\'e polynomial of the stable homology $H_i(\Br_n;H_1(\surf_n;\Z))$ as a $\Z_2$-module is the following:
$$
P_2(\Br;H_1(\Sigma))(q) = \frac{q}{1-q^2} \prod_{j \geq 1} \frac{1}{1-q^{2^j-1}}
$$
\end{thm}
An explicit computation of the first terms of the stable series $P_2(\Br;H_1(\Sigma))(q)$ gives
$$
q+q^2+ 2q^3 + 3q^4 + 4q^5 + 5q^6 + 7q^7 + 9 q^8 + 11q^9 + 14q^{10} + 17 q^{11} + \ldots
$$


\bibliographystyle{alpha}
\bibliography{biblio} 

\end{document}